\documentclass[11 pt]{amsart}
\usepackage[margin=3.3 cm]{geometry}
\usepackage{amsfonts}
\usepackage{amsmath}
\usepackage{amsthm}
\usepackage{amssymb}
\usepackage{amsmath,amscd}
\usepackage{graphicx}
\usepackage{enumerate}
\usepackage{subcaption}
\graphicspath{ {./Figure/} }
\usepackage[pdffitwindow=false,
colorlinks=true,linkcolor=black,urlcolor=black,citecolor=black]{hyperref}
\usepackage{tabularx}
\usepackage{float}
\usepackage{multirow}

\newtheorem{theorem}{Theorem}[section]
\newtheorem{lemma}[theorem]{Lemma}
\theoremstyle{definition}
\newtheorem{definition}[theorem]{Definition}

\newtheorem{proposition}[theorem]{Proposition}
\newtheorem{corollary}[theorem]{Corollary}

\newtheorem{remark}[theorem]{Remark}
\numberwithin{equation}{section}

\newcommand{\spin}{\mathrm{Spin}}
\newcommand{\trace}{\mathrm{tr}}
\newcommand{\ricci}{\mathrm{Ric}}

\title{Non-Shrinking Ricci Solitons of cohomogeneity one from the quaternionic Hopf fibration}
\author{Hanci Chi}
\address{Department of Foundational Mathematics\\ Xi'an Jiaotong-Liverpool University\\ Suzhou 215123\\ China}
\email{hanci.chi@xjtlu.edu.cn}
\begin{document}
\maketitle
\begin{abstract}
We establish the existence of two 3-parameter families of non-Einstein, non-shrinking Ricci solitons: one on $\mathbb{H}^{m+1}$ and one on $\mathbb{HP}^{m+1}\backslash\{*\}$. Each family includes a continuous 1-parameter subfamily of asymptotically paraboloidal (non-collapsed) steady Ricci solitons, with the Jensen sphere as the base. Additionally, we extend this result by proving the existence of a 2-parameter family on $\mathbb{O}^2$, which contains a 1-parameter subfamily of asymptotically paraboloidal steady Ricci solitons based on the Bourguignon--Karcher sphere.
\end{abstract}

\let\thefootnote\relax\footnote{


The author is supported by NSFC (No. 12071489, No. 12301078), the Natural Science Foundation of Jiangsu Province (BK-20220282), and XJTLU Research Development Funding (RDF-21-02-083).}

\section{Introduction}
A Riemannian manifold $(M,g)$ is a \emph{Ricci soliton} if for a vector field $X$ and a constant $\epsilon$, the equation
\begin{equation}
\label{eqn_RS}
\ricci+\frac{1}{2}\mathcal{L}_Xg+\frac{\epsilon}{2}g=0
\end{equation}
is satisfied. A Ricci soliton is a self-similar solution to the Ricci flow. Specifically, a Ricci soliton homothetically expands ($\epsilon>0$), shrinks ($\epsilon<0$) or remains steady ($\epsilon=0$) under the Ricci flow. If $X=\mathrm{grad}(f)$ for some potential function $f$, the Ricci soliton is called a \emph{gradient Ricci soliton}.
An Einstein manifold is automatically a Ricci soliton with a constant potential function. Ricci solitons are candidates for blow-ups of singularities of the Ricci flow. 

Many known Ricci solitons are of \emph{cohomogeneity one}, i.e., a Lie group acts isometrically on $(M,g)$ such that the principal orbit is of codimension one. 
Under such a symmetry condition, the Ricci soliton equation is reduced to a set of ODEs. The initial value problem for cohomogeneity one Einstein equation was systematically studied in \cite{eschenburg_initial_2000} and was later generalized to the cohomogeneity one Ricci soliton equation in \cite{buzano_initial_2011}. A general theory for the global existence problem is yet to be established. 

This paper investigates the existence of non-shrinking, non-Einstein Ricci solitons of cohomogeneity one. Previous examples include the Bryant soliton \cite{bryant_ricci_1992} on $\mathbb{R}^3$, which is rotationally symmetric and easily generalizes to $\mathbb{R}^n$. The Bryant soliton is extended to non-shrinking Ricci solitons on trivial vector bundles over multiple warped products in \cite{ivey_new_1994}, \cite{dancer_non-kahler_2009}, \cite{dancer_new_2009}, \cite{buzano_family_2015}, \cite{buzano_non-kahler_2015} and \cite{nienhaus_new_2023}. The existence of K\"ahler Ricci solitons on the canonical line bundle over $\mathbb{CP}^n$ ($n\geq 1$) was established in \cite{koiso1990rotationally}, \cite{Cao_Elliptic_1996} and \cite{feldman_rotationally_2003}. Such an example was further generalized in \cite{dancer2011ricci} to the total space of certain complex vector bundles over a product of Fano manifolds. Without the K\"ahler condition, it was shown independently in  \cite{stolarski_steady_2015}, \cite{appleton_family_2017}, and \cite{wink_cohomogeneity_2017} that a broader family of Ricci solitons exist on complex line bundles over $\mathbb{CP}^n$. It was further proved in \cite{appleton_family_2017} that some of these complex line bundles admit a non-collapsed steady soliton.

In \cite{wink_complete_2021} and \cite{wink_cohomogeneity_2017}, examples of non-shrinking Ricci solitons were found on cohomogeneity one vector bundles, whose principal orbits are Hopf fibrations. These solitons are of two summands, meaning their collapsing spheres and base spaces are irreducible. This paper generalizes the Wink solitons by considering the quaternionic Hopf fibration with three summands. In particular, the group triple under consideration is
\begin{equation}
\label{eqn_group triple}
(\mathsf{K},\mathsf{H},\mathsf{G})=(Sp(m)\triangle U(1),Sp(m)Sp(1)U(1),Sp(m+1)U(1)).
\end{equation}
The isotropy representation of $\mathsf{G}/\mathsf{K}$ splits into three irreducible summands, leading to a broader family of non-shrinking cohomogeneity one Ricci solitons on $\mathbb{HP}^{m+1}\backslash\{*\}$ and $\mathbb{H}^m$. Furthermore, this paper extends the existence results from \cite{chi2021einstein}, where the cohomogeneity one Einstein equation associated to \eqref{eqn_group triple} was examined. While the Einstein case exhibits various asymptotic behaviors, including asymptotically conical (AC), locally asymptotically conical (ALC), and asymptotically hyperbolic (AH), the newly discovered non-Einstein steady Ricci solitons are either asymptotically paraboloidal (non-collapsed) or asymptotically cigar-paraboloidal, as defined below.

\begin{definition}
Let $(M,g_M)$ be a Riemannian manifold of dimension $n+1$. Let $(N,g_N)$ be an $n$-dimensional Riemannian manifold and $(P(N),dt^2 +tg_N)$ be the metric paraboloid with base $N$. Let $\bullet$ denote the tip of the paraboloid. The manifold $M$ is \emph{asymptotically paraboloidal} (AP) if for some $p\in M$, we have $\lim\limits_{l\to\infty}((M, p), \frac{1}{l}g_M) = ((P(N),\bullet),dt^2+t g_N)$ in the pointed Gromov–Hausdorff sense.
\end{definition}

\begin{definition}
Let $(M,g_M)$ be a Riemannian manifold of dimension $n+2$. Let $(N,g_N)$ be an $n$-dimensional Riemannian manifold and $(P(N),dt^2 +tg_N)$ be the metric paraboloid with base $N$. Let $\bullet$ denote the tip of the paraboloid. The manifold $M$ is \emph{asymptotically cigar-paraboloidal} (ACP) if for some $p\in M$, we have $\lim\limits_{l\to\infty}((M, p), \frac{1}{l}g_M) = ((\hat{P}(N),\bullet),dt^2+Ads^2+t g_N)$ in the pointed Gromov–Hausdorff sense, where $\hat{P}(N)$ is a circle bundle over $P(N)$ and $A$ is a constant.
\end{definition}

Our main theorems are the following.
\begin{theorem}
\label{thm_1}
There exists a continuous 3-parameter family of complete $Sp(m+1)U(1)$-invariant Ricci solitons $\{\zeta(s_1,s_2,s_3,s_4)\mid (s_1,s_2,s_3,s_4)\in \mathbb{S}^3, s_1,s_4>0,s_2,s_3\geq 0\}$ on $\mathbb{HP}^{m+1}\backslash\{*\}$.
\begin{enumerate}
\item
Each $\zeta(s_1,s_2,0,s_4)$ is a non-Einstein steady Ricci soliton. The soliton is AP if $s_2=0$, with the base of the limit paraboloid being the Jensen $\mathbb{S}^{4m+3}$. The soliton is ACP if $s_2>0$, with the base of the limit paraboloid being the non-K\"ahler $\mathbb{CP}^{2m+1}$.
\item
Each $\zeta(s_1,s_2,s_3,s_4)$ with $s_3>0$ is an AC non-Einstein expanding Ricci soliton.
\end{enumerate}
\end{theorem}

\begin{theorem}
\label{thm_2}
There exists a continuous 3-parameter family of complete $Sp(m+1)U(1)$-invariant Ricci solitons $\{\gamma(s_1,s_2,s_3,s_4)\mid (s_1,s_2,s_3,s_4)\in \mathbb{S}^3, s_1,s_2,s_3\geq 0,s_4> 0\}$ on $\mathbb{H}^{m}$.
\begin{enumerate}
\item
Each $\gamma(s_1,s_2,0,s_4)$ is a non-Einstein steady Ricci soliton. The soliton is AP if $s_1> 0$ and $s_2=0$, with the base of the limit paraboloid being the Jensen $\mathbb{S}^{4m+3}$. The soliton is ACP if $s_1=0$ and $s_2>0$, with the base of the limit paraboloid being the Fubini--Study $\mathbb{CP}^{2m+1}$. The soliton is also ACP if $s_1, s_2>0$, with the base of the limit paraboloid being the non-K\"ahler $\mathbb{CP}^{2m+1}$.
\item
Each $\gamma(s_1,s_2,s_3,s_4)$ with $s_3>0$ is an AC non-Einstein expanding Ricci soliton.
\end{enumerate}
\end{theorem}

In the theorems above, scaling the vector $(s_1,s_2,s_3,s_4)$ corresponds to applying a homothetic change to the metrics. Parameters $s_1$ and $s_2$ determine the initial squashing of the principal orbit. These parameters are derived by following the initial value theorem in \cite{eschenburg_initial_2000}. The parameter $s_3$ controls the generalized mean curvature of the principal orbit. If $\epsilon=0$, the generalized mean curvature is freely scaled under homothetic change, making the parameter $s_3$ trivial. We thus set $s_3=0$ for the steady case. From \cite{buzano_initial_2011}, we have the parameter $s_4>0$ for non-Einstein Ricci solitons. The parameter is related to the initial condition $\ddot{f}(0)$. Known examples of cohomogeneity one Ricci solitons recovered by Theorem \ref{thm_1}-\ref{thm_2} are summarized below.

\small
\begin{table}[H]
\centering
  \begin{tabular}{l l l l| l }
 & Metric & Asymptotic & Base & Source \\
\hline
\hline
   $\zeta(1,0,0,0)$, $m=1$ & $\spin(7)$ (Ricci-flat)  & AC & Jensen $\mathbb{S}^7$&\cite{bryant_construction_1989}\\ 
\hline
   $\zeta(s_1,s_2,0,0)$, $m=1$ & $\spin(7)$ (Ricci-flat) & ALC  & nearly-K\"ahler $\mathbb{CP}^3$&\cite{cvetic_cohomogeneity_2002} \\ 
\hline
$\zeta(1,0,0,0)$, $m>1$  &Ricci-flat        & AC   & Jensen $\mathbb{S}^{4m+3} $& \multirow{2}{*}{\cite{bohm_non-compact_1999}}\\
\cline{1-4}
   $\zeta(s_1,0,s_3,0)$     &negative Einstein & AH  &&\\ 
       \hline
   $\zeta(s_1,s_2,0,0)$ ,  $m>1$     & Ricci-flat         & ALC &non-K\"ahler $\mathbb{CP}^{2m+1}$&  \multirow{2}{*}{\cite{chi2021einstein} } \\
\cline{1-4}
   $\zeta(s_1,s_2,s_3,0)$   & negative Einstein  & AH & & \\ 
       \hline
   $\zeta(s_1,0,0,s_4)$     & steady Ricci solitons    & AP & Jensen $\mathbb{S}^{4m+3}$ & \multirow{2}{*}{\cite{wink_cohomogeneity_2017}}\\
\cline{1-4}
   $\zeta(s_1,0,s_3,s_4)$   & expanding Ricci solitons & AC && \\
       \hline
   $\zeta(s_1,s_2,0,s_4)$   & steady Ricci solitons    & ACP & non-K\"ahler $\mathbb{CP}^{2m+1}$&   \multirow{2}{*}{present paper}\\ 
   \cline{1-4}
      $\zeta(s_1,s_2,s_3,s_4)$ & expanding Ricci solitons & AC & &   \\ 
       \hline
  \end{tabular}
  \caption{}
  \label{tab_zeta_curve}
\end{table}

\begin{table}[H]
\centering
  \begin{tabular}{l l l l |l}
 & Metric & Asymptotic & Base & Source \\
\hline
\hline
   $\gamma(0,0,0,0)$ & flat && & \\ 
          \hline
   $\gamma(0,1,0,0)$ & Ricci-flat  & ALC & Fubini--Study $\mathbb{CP}^{2m+1}$ &\multirow{2}{*}{\cite{berard-bergery_sur_1982}} \\ 
   $\gamma(0,s_2,s_3,0)$ & negative Einstein  & AH & \\ 
          \hline
   $\gamma\left(\frac{1}{\sqrt{5}},\frac{2}{\sqrt{5}},0,0\right)$ & \multirow{2}{*}{$\spin(7)$ (Ricci-flat)} & \multirow{2}{*}{ALC} & \multirow{2}{*}{nearly-K\"ahler $\mathbb{CP}^3$} & \multirow{2}{*}{\cite{cvetic_cohomogeneity_2002}}\\
   $m=1$&&&&
   \\
   \hline  
   $\gamma(1,0,0,0)$ & Ricci-flat & AC & Jensen $\mathbb{S}^{4m+3}$ &  \multirow{3}{*}{\cite{chi2021einstein}}\\ 
      \cline{1-4}
   $\gamma(s_1,s_2,0,0)$ & Ricci-flat & ALC & non-K\"ahler $\mathbb{CP}^{2m+1}$ &\\ 
      \cline{1-4}
   $\gamma(s_1,s_2,s_3,0)$ & negative Einstein & AH & \\ 
       \hline
   $\gamma(0,0,0,1)$ & steady Ricci soliton & AP & standard $\mathbb{S}^{4m+3}$ &\multirow{2}{*}{\cite{bryant_ricci_1992}}\\ 
      \cline{1-4}
   $\gamma(0,0,s_3,s_4)$ & expanding Ricci-soliton & AC & &\\ 
       \hline
   $\gamma(0,\sigma,0,(n+3)\sigma)$ & \multirow{2}{*}{steady K\"ahler--Ricci soliton} & \multirow{2}{*}{ACP} & \multirow{2}{*}{Fubini--Study $\mathbb{CP}^{2m+1}$} & \multirow{2}{*}{\cite{Cao_Elliptic_1996}}\\ 
    $\sigma=\frac{1}{\sqrt{n^2+6n+10}}$&&&&\\  
   \hline
   $\gamma(0,s_2,s_3,s_4)$ & \multirow{2}{*}{expanding K\"ahler--Ricci soliton} & \multirow{2}{*}{AC} &  &\multirow{2}{*}{\cite{cao1997limits}}\\
   $s_4=(n+3)s_2+\epsilon s_3$&&&&  \\ 
         \hline
   $\gamma(s_1,0,0,s_4)$ & steady Ricci soliton & AP & Jensen $\mathbb{S}^{4m+3}$& \multirow{3}{*}{present paper}\\ 
      \cline{1-4}
   $\gamma(s_1,s_2,0,s_4)$ & steady Ricci soliton  & ACP& non-K\"ahler $\mathbb{CP}^{2m+1}$&\\ 
      \cline{1-4}
   $\gamma(s_1,s_2,s_3,s_4)$ & expanding Ricci soliton  & AC & \\ 
          \hline
  \end{tabular}
    \label{tab_gamma_curve}
  \caption{}
\end{table}

The analytic techniques in this paper are generalized to the case where the principal orbit is the octonionic Hopf fibration. The Wink solitons on $\mathbb{OP}^2\backslash\{*\}$ are thus recovered. In addition, new non-shrinking Ricci solitons are found on $\mathbb{O}^2$. Since $\mathbb{S}^7$ and $\mathbb{OP}^1$ are irreducible, there is only one squashing parameter $s_1$.
\begin{theorem}
\label{thm_3}
There exists a continuous 2-parameter family of complete $\spin(9)$-invariant Ricci solitons $\{\tilde{\gamma}(s_1,s_3,s_4)\mid (s_1,s_3,s_4)\in \mathbb{S}^2, s_1,s_3\geq 0,s_4> 0\}$ of cohomogeneity one on $\mathbb{O}^2$. Specifically,
\begin{enumerate}
\item
Each $\tilde{\gamma}(s_1,0,s_4)$ is a steady Ricci solitons. If $s_4=0$, the soliton is Ricci-flat with an AC asymptotic, with the base of the limit cone being the Bourguignon--Karcher $\mathbb{S}^{15}$. If $s_4> 0$, the soliton is non-Einstein with an AP asymptotic, with the base of the limit paraboloid being the Bourguignon--Karcher $\mathbb{S}^{15}$.
\item
Each $\tilde{\gamma}(s_1,s_3,s_4)$ is an expanding Ricci soliton. If $s_4=0$, the soliton is Einstein with an AH asymptotic. If $s_4>0$. the soliton is an AC non-Einstein Ricci soliton.
\end{enumerate}
\end{theorem}

\begin{remark}
Let $n=\dim{\mathsf{G}/\mathsf{K}}$. The hypersurfaces of each AP steady Ricci soliton in Theorem \ref{thm_1}-\ref{thm_3} has volume growth rate of $t^{\frac{n}{2}}$. We refrain from referring to this as Euclidean volume growth since the principal orbits are non-standard homogeneous spheres at infinity. The hypersurface of each ACP steady Ricci soliton in Theorem \ref{thm_1}-\ref{thm_2} has volume growth rate of $t^{\frac{n-1}{2}}$. For the same reason that the base of the asymptotic cone is not a standard sphere, we do not classify the AC Ricci-flat metric in Theorem \ref{thm_3} as asymptotically Euclidean or asymptotically locally Euclidean.
\end{remark}

\begin{remark}
Replace the intermediate group $\mathsf{H}$ by $\mathsf{L}=Sp(m)U(1)U(1)$. The group triple $(\mathsf{K},\mathsf{L},\mathsf{G})$ gives rise to the complex Hopf fibration. If $\mathsf{G}/\mathsf{K}$ collapses to $\mathsf{G}/\mathsf{L}=\mathbb{CP}^{2m+1}$, the total space is a complex line bundle over $\mathbb{CP}^{2m+1}$. We refer to \cite{chi2024infinitely} for a study of this case, as the analytic techniques involved are different.
\end{remark}

By \cite{bryant_ricci_1992} and \cite{cao1997limits}, Ricci solitons represented by $\gamma(0,0,s_3,s_4)$ and $\gamma(0,s_2,s_3,(n+3)s_2+\epsilon s_3)$ have positive sectional curvature. We show that $\gamma(s_1,0,0,s_4)$ contains new non-collapsed steady Ricci soliton with positive sectional curvatures, arising as small perturbations of the Bryant soliton.

\begin{theorem}
\label{thm_4}
There exists a sufficiently small $\delta>0$ such that each $\gamma(s_1,0,0,\sqrt{1-s_1^2})$ (resp. $\tilde{\gamma}(s_1,0,\sqrt{1-s_1^2})$) with $s_1\in [0,\delta)$ is a positively curved Ricci soliton on $\mathbb{H}^{4m+4}$ (resp. $\mathbb{O}^2$), whose asymptotic paraboloid has a base as the Jensen sphere (resp. the Bourguignon--Karcher sphere).
\end{theorem}

Since the base of the  asymptotic paraboloid is not the standard sphere, Ricci solitons in the conjecture above are not asymptotically cylindrical in the Brendle sense. Any positively curved asymptotically cylindrical steady Ricci soliton is isometric to the Bryant soliton up to scaling; see \cite{brendle2014rotational}.

The paper is organized as follows: First, we describe the cohomogeneity one Ricci soliton equation and apply the coordinate change from \cite{dancer_non-kahler_2009}. Next, we prove the existence of complete Ricci solitons by extending the compact invariant set in \cite{chi2021einstein}. We then investigate the asymptotic behavior of these Ricci solitons, proving the existence of a 1-parameter family of AP steady Ricci solitons on $\mathbb{H}^{m+1}$ and extend our study to the case of $\mathbb{O}^2$. Finally, we prove Theorem \ref{thm_4}. To provide a clearer illustration of our results, we present Figure \ref{fig_1} and Figure \ref{fig_2} at the end.

\section{Cohomogeneity one Ricci soliton equation}
In this section, we restrict to the case $\epsilon\geq 0$ and write down the corresponding cohomogeneity one system. Let $Q$ be the inner product on $\mathfrak{g}$ that generates the standard metric on $\mathbb{S}^{4m+3}$. As a $\mathsf{K}$-module, the isotropy representation $\mathfrak{g}/\mathfrak{k}$ splits into three $Q$-orthogonal irreducible summands $\mathbf{1}\oplus \mathbf{2}\oplus \mathbf{4m}$, where $\mathfrak{h}/\mathfrak{k}=\mathbf{1}\oplus \mathbf{2}$ and $\mathfrak{g}/\mathfrak{h}=\mathbf{4m}$. 
By \cite{ziller_homogeneous_1982}, the scalar curvature of $\mathsf{G}/\mathsf{K}$ is 
\begin{equation}
\label{eqn_scalar curvature for G/K}
r_s=\frac{8}{b^2}+\frac{4m(4m+8)}{c^2}-2\frac{a^2}{b^4}-4m\frac{a^2}{c^4}-8m\frac{b^2}{c^4}.
\end{equation}
Consider the cohomogeneity one metric
\begin{equation}
\label{eqn_coho1metric}
g=dt^2+a^2(t)\left.Q\right|_{\mathbf{1}}+b^2(t)\left.Q\right|_{\mathbf{2}}+c^2(t)\left.Q\right|_{\mathbf{4m}}.
\end{equation}
By \cite{eschenburg_initial_2000} and \cite{buzano_initial_2011}, the cohomogeneity one non-shrinking Ricci soliton equation is
\begin{equation}
\label{eqn_cohomogeneity_one_soliton_equation}
\begin{split}
\frac{\ddot{a}}{a}-\left(\frac{\dot{a}}{a}\right)^2&=-\left(\frac{\dot{a}}{a}+2\frac{\dot{b}}{b}+4m\frac{\dot{c}}{c}\right)\frac{\dot{a}}{a}+2\frac{a^2}{b^4}+4m\frac{a^2}{c^4}+\frac{\dot{a}}{a}\dot{f}+\frac{\epsilon}{2},\\
\frac{\ddot{b}}{b}-\left(\frac{\dot{b}}{b}\right)^2&=-\left(\frac{\dot{a}}{a}+2\frac{\dot{b}}{b}+4m\frac{\dot{c}}{c}\right)\frac{\dot{b}}{b}+\frac{4}{b^2}-2\frac{a^2}{b^4}+4m\frac{b^2}{c^4}+\frac{\dot{b}}{b}\dot{f}+\frac{\epsilon}{2},\\
\frac{\ddot{c}}{c}-\left(\frac{\dot{c}}{c}\right)^2&=-\left(\frac{\dot{a}}{a}+2\frac{\dot{b}}{b}+4m\frac{\dot{c}}{c}\right)\frac{\dot{c}}{c}+\frac{4m+8}{c^2}-2\frac{a^2}{c^4}-4\frac{b^2}{c^4}+\frac{\dot{c}}{c}\dot{f}+\frac{\epsilon}{2},\\
\ddot{f}+\frac{\epsilon}{2}&=\frac{\ddot{a}}{a}+2\frac{\ddot{b}}{b}+4m\frac{\ddot{c}}{c}.
\end{split}
\end{equation}

If we redefine $a\to a$, $b\to \sqrt{2}b$, $c\to 2c$, and set $f\equiv 0$, we obtain the cohomogeneity one Einstein equation in \cite{chi2021einstein}. For the steady case, the system \eqref{eqn_cohomogeneity_one_soliton_equation} is invariant under homothetic change. For the expanding case, we fix the homothety by setting $\epsilon=1$.

Let $\bar{r}_s$ be the scalar curvature of the ambient space. Using the general identity 
\begin{equation}
\label{eqn_original_conservation_law}
\bar{r}_s+|\nabla f|^2+\epsilon f=\tilde{C}
\end{equation}
in \cite{hamilton_formations_1993}, we obtain the conservation law for the potential function
\begin{equation}
\label{eqn_potential_equation}
\ddot{f}+(\trace(L)-\dot{f})\dot{f}-\epsilon f=C,
\end{equation}
where $C = -\frac{n+1}{2}\epsilon - \tilde{C}$ is a constant, and $L$ is the shape operator of $\mathsf{G}/\mathsf{K}$.
Canceling all second derivative terms using \eqref{eqn_cohomogeneity_one_soliton_equation} and \eqref{eqn_potential_equation}, we have 
\begin{equation}
\label{eqn_cohomogneity_one_conservation_law}
\begin{split}
&\trace(L^2)+r_s+\frac{n-1}{2}\epsilon-(\dot{f}-\trace(L))^2=C+\epsilon f.
\end{split}
\end{equation}
Using the Gauss--Codazzi equation, we write $\bar{r}_s$ in terms of the potential function:
\begin{equation}
\label{eqn_scalar curvature in terms of f}
\bar{r}_s=-\frac{n+1}{2}\epsilon-(C+\epsilon f)-(\dot{f})^2.
\end{equation}
It is known that $\bar{r}_s> -\frac{n+1}{2}\epsilon$ for a non-Einstein Ricci soliton; see \cite[Corollary 2.3]{chen_strong_2009}. Thus, equation \eqref{eqn_scalar curvature in terms of f} implies $C + \epsilon f < 0$. In particular, the constant $C$ must be negative if the non-Einstein Ricci soliton is steady.

Adding a constant to the potential function $f$ does not change the Ricci soliton equation. Hence, we set $f(0)=0$ without loss of generality. By \cite{buzano_initial_2011}, the potential function is an even function of $t$ near $0$. Thus, we set $\dot{f}(0)=0$ so that the function extended smoothly to the singular orbit. Let $d_S$ denote the dimension of the collapsing sphere, where either $d_S = 3$ or $d_S = 4m + 3$. We have $\trace(L) = \frac{d_S}{t} + O(t)$ as $t \to 0$. It follows from \eqref{eqn_potential_equation} that $\ddot{f}(0) = \frac{C}{d_S+1}$. The initial condition for the potential function is established.

We consider the following two initial conditions for metric components. If $\mathsf{G}/\mathsf{K}$ collapses to $\mathsf{G}/\mathsf{H}=\mathbb{HP}^{m}$ at $t=0$, we have
\begin{equation}
\label{eqn_initial_condition_G/H}
\lim_{t\to 0} (a,b,c,\dot{a},\dot{b},\dot{c})=(0,0,c_0,1,1,0),\quad c_0> 0.
\end{equation}
There exists another parameter from $(\dddot{a}-\dddot{b})(0)$; see \cite[Proposition 2.4]{chi2021einstein}. The $\mathbb{S}^3$-fiber is squashed if such a parameter is nonzero. If the principal orbit fully collapses to a point, the initial condition is
\begin{equation}
\label{eqn_initial_condition_Taub-NUT}
\lim_{t\to 0} (a,b,c,\dot{a},\dot{b},\dot{c})=(0,0,0,1,1,1).
\end{equation}
For this case, there are two parameters from $(\dddot{a}-\dddot{b})(0)$ and $(\dddot{b}-\dddot{c})(0)$; see \cite[Remark 2.5]{chi2021einstein}. If these parameters are nonzero, the principal orbit collapses as the standard $\mathbb{S}^{4m+3}$ only up to the second order.


If $\epsilon=1$, both initial conditions have one free parameter from the potential function and two others from the homogeneous metric components. If $\epsilon = 0$, the cohomogeneity one equation becomes invariant under homothetic change. Thus, there is one less free parameter in the steady case. 

Apply coordinate change $d\eta=(\trace(L)-\dot{f})dt$. Define
\begin{equation}
\label{eqn_new_variable_for_new_coordinate}
\begin{split}
&X_1=\frac{\frac{\dot{a}}{a}}{\trace(L)-\dot{f}},\quad X_2=\frac{\frac{\dot{b}}{b}}{\trace(L)-\dot{f}},\quad X_3=\frac{\frac{\dot{c}}{c}}{\trace(L)-\dot{f}},\\
&Y_1=\frac{a}{b},\quad Y_2=\frac{\frac{1}{b}}{\trace(L)-\dot{f}},\quad Y_3=\frac{\frac{b}{c^2}}{\trace(L)-\dot{f}},\quad W=\frac{1}{(\trace(L)-\dot{f})^2}.
\end{split}
\end{equation}


Define the following functions on $\eta$.
\begin{equation}
\label{eqn_curvature function}
\begin{split}
&R_1=2 Y_1^2Y_2^2+4m Y_1^2Y_3^2,\quad R_2=4Y_2^2-2 Y_1^2Y_2^2+4m Y_3^2,\quad R_3=(4m+8)Y_2Y_3-2Y_1^2Y_3^2-4Y_3^2,\\
&R_s=R_1+2R_2+4mR_3,\quad H=X_1+2X_2+4mX_3,\quad G=X_1^2+2X_2^2+4mX_3^2.
\end{split}
\end{equation}
Let $'$ denote the derivative with respect to $\eta$. The soliton equation \eqref{eqn_cohomogeneity_one_soliton_equation} is transformed to
\begin{equation}
\label{eqn_Polynomial_soliton_equation}
\begin{bmatrix}
X_1\\
X_2\\
X_3\\
Y_1\\
Y_2\\
Y_3\\
W
\end{bmatrix}'=
V(X_1,X_2,X_3,Y_1,Y_2,Y_3,W):=\begin{bmatrix}
X_1(G-\frac{\epsilon}{2}W-1)+R_1+\frac{\epsilon}{2}W\\
X_2(G-\frac{\epsilon}{2}W-1)+R_2+\frac{\epsilon}{2}W\\
X_3(G-\frac{\epsilon}{2}W-1)+R_3+\frac{\epsilon}{2}W\\
Y_1(X_1-X_2)\\
Y_2(G-\frac{\epsilon}{2}W-X_2)\\
Y_3(G-\frac{\epsilon}{2}W+X_2-2X_3)\\
2W(G-\frac{\epsilon}{2}W)
\end{bmatrix}.
\end{equation}
Define $Q:= G+R_s+(n-1)\frac{\epsilon}{2}W-1$.
The conservation law \eqref{eqn_cohomogneity_one_conservation_law} becomes
\begin{equation}
\label{eqn_new_cohomogneity_one_conservation_law}
\begin{split}
Q=\frac{C+\epsilon f}{(\trace(L)-\dot{f})^2}.
\end{split}
\end{equation}
As discussed above, the quantity $C+\epsilon f$ remains negative for a non-Einstein Ricci soliton. Therefore, trajectories representing non-Einstein Ricci solitons are characterized by the inequalities $Q<0$ and $H<1$. In particular, Proposition 3.1 in \cite{dancer_non-kahler_2009} can be adapted to our context, as demonstrated below. 
\begin{proposition}
\label{prop_phase_space_for_Ricci_solitons}
The set 
$$\mathcal{RS}:=\{Q\leq 0\}\cap\{H\leq 1\}\cap\{W\geq 0\}\cap \{Y_1,Y_2,Y_2\geq 0\}$$ is flow-invariant.
\end{proposition}
\begin{proof}
Straightforward computations show that
\begin{equation}
\label{eqn_characteriaztion_of_non_trivial_soliton}
\begin{split}
\begin{bmatrix}
Q\\
H
\end{bmatrix}'=\begin{bmatrix}
2Q(G-\frac{\epsilon}{2}W)+\epsilon(H-1)W\\
(H-1)\left(G-\frac{\epsilon}{2}W-1\right)+Q
\end{bmatrix}.
\end{split}
\end{equation}
The derivatives of $W$ and $Y_i$ show that $\mathcal{RS}$ is flow-invariant.
\end{proof}

Based on the above proposition, trajectories representing Einstein metrics are characterized by the equalities $Q=0$ and $H=1$. As $\trace(L)-\dot{f}$ can be scaled freely under homothetic change for the steady case, it is unsurprising that $W$ is decoupled if $\epsilon=0$. Trajectories representing steady Ricci solitons can be characterized by the equality $W=0$ even though the function $\frac{1}{(\trace(L)-\dot{f})^2}$ does not vanish. We summarize the above discussion with the following table.
\begin{center}
\begin{table}[H]
\begin{tabular}{l l l}
\text{Metric}&\text{Dim. of phase space}& \text{Equations and inequalites}\\
\hline
\hline
\text{Ricci-flat} &4&  $Q=0$, $H=1$, $W=0$\\ 
\hline
\text{Negative Einstein}&5&  $Q=0$, $H=1$, $W>0$ \\ 
\hline
\text{Non-Einstein steady Ricci soliton} &6&  $Q<0$, $H<1$, $W=0$\\ 
\hline
\text{Non-Einstein expanding Ricci soliton} &7&  $Q<0$, $H<1$, $W>0$\\ 
\hline
\end{tabular}
\caption{}
\end{table}
\label{tab_dimension of phase space}
\end{center}
Besides $\mathcal{RS}_{\text{Einstein}}:=\mathcal{RS}\cap\{Q=0\}\cap\{H=1\}$ and $\mathcal{RS}_{\text{steady}}:=\mathcal{RS}\cap\{W=0\}$, there are two other invariant subsets with geometric importance.
\begin{enumerate}
\item
$\mathcal{RS}_{\text{FS}}:=\mathcal{RS}\cap\{Y_2-Y_3=0\}\cap\{X_2-X_3=0\}$.

The principal orbit is a circle bundle over the Fubini--Study $\mathbb{CP}^{2m+1}$.
\item
$\mathcal{RS}_{\text{KE}}:=\mathcal{RS}_{\text{FS}}\cap \left\{(4m+4)Y_2Y_3+\frac{\epsilon}{2}W-X_2(1+X_1)=0\right\}\cap \{X_2-Y_1Y_2=0\}$

Integral curves on this invariant set represent K\"ahler--Ricci solitons; see \cite{dancer2011ricci}.
\item
$\mathcal{RS}_{\text{round}}:=\mathcal{RS}\cap\{Y_1=1\}\cap\{X_1-X_2=0\}$.

The principal orbit is a round $\mathbb{S}^3$-bundle over $\mathbb{HP}^{m}$.
\end{enumerate}

For $\epsilon=1$, the variable $t$ and functions $a$, $b$, $c$ are recovered by
\begin{equation}
\label{eqn_recovery_expanding}
t=\int_{\eta_0}^\eta \sqrt{W} d\tilde{\eta}+t(\eta_0), \quad a=\frac{Y_1 \sqrt{W}}{Y_2},\quad b=\frac{\sqrt{W}}{Y_2},\quad c=\sqrt{\frac{W}{Y_2Y_3}}.
\end{equation}
For $\epsilon=0$, we recover variable $t$ and functions $a$, $b$, $c$ by \eqref{eqn_recovery_expanding} with $\sqrt{W}$ replaced by 
\begin{equation}
\label{eqn_recovery_steady}
\tilde{W}:= \tilde{W}(\eta_0)\exp\left(\int_{\eta_0}^\eta G d\tilde{\eta}\right).
\end{equation}

We end this section with the following proposition, which transforms the existence problem of a complete Ricci soliton to finding an integral curve to \eqref{eqn_Polynomial_soliton_equation} that is defined on $\mathbb{R}$.
\begin{proposition}
\label{prop_etainfty is tinfty}
For $\epsilon=0$, the solution to the original system \eqref{eqn_cohomogeneity_one_soliton_equation} is defined on $(0,\infty)$ if the corresponding integral curve to \eqref{eqn_Polynomial_soliton_equation} is defined on $\mathbb{R}$. The statement holds for $\epsilon=1$ if $\lim\limits_{\eta\to\infty} W=0$.
\end{proposition}
\begin{proof}
For $\epsilon = 0$, since eventually $\tilde{W}>\delta$ for some $\delta>0$ from \eqref{eqn_recovery_steady}, we have $\lim\limits_{\eta\to\infty}t=\infty$ from \eqref{eqn_recovery_expanding}. The proof is complete. For $\epsilon=1$, the statement follows by Lemma 3.14 in \cite{dancer_non-kahler_2009}. 
\end{proof}

\section{Existence of complete Ricci solitons}
\label{sec_Global_existence}
In the new coordinates \eqref{eqn_new_variable_for_new_coordinate}, initial conditions \eqref{eqn_initial_condition_G/H} and \eqref{eqn_initial_condition_Taub-NUT} are respectively transformed to critical points
\begin{equation}
\begin{split}
&P_{\mathbb{HP}^{m}}:=\left(\frac{1}{3},\frac{1}{3},0,1,\frac{1}{3},0,0\right), P_{\bullet}=\left(\frac{1}{n},\frac{1}{n},\frac{1}{n},1,\frac{1}{n},\frac{1}{n},0\right).
\end{split}
\end{equation}
We compute linearizations at these points and establish the local existence of Ricci solitons. The linearization at $P_{\mathbb{HP}^{m}}$ is
\begin{equation}
\begin{bmatrix}
-\frac{4}{9}&\frac{4}{9}&0&\frac{4}{9}&\frac{4}{3}&0&\frac{\epsilon}{3}\\
\frac{2}{9}&-\frac{2}{9}&0&-\frac{4}{9}&\frac{4}{3}&0&\frac{\epsilon}{3}\\
0&0&-\frac{2}{3}&0&0&\frac{4m+8}{3}&\frac{\epsilon}{2}\\
1&-1&0&0&0&0&0\\
\frac{2}{9}&\frac{1}{9}&0&0&0&0&-\frac{\epsilon}{6}\\
0&0&0&0&0&\frac{2}{3}&0\\
0&0&0&0&0&0&\frac{2}{3}\\
\end{bmatrix}.
\end{equation}
The critical point is hyperbolic. Unstable eigenvectors all have the same eigenvalue $\frac{2}{3}$. These eigenvectors are
\begin{equation}
\label{eqn_unstable_eigen_zeta}
\begin{split}
&u_1=\begin{bmatrix}
-4m(m+2)\\
-4m(m+2)\\
3(m+2)\\
0\\
-2m(m+2)\\
3\\
0
\end{bmatrix},
u_2=\begin{bmatrix}
-4\\
2\\
0\\
-9\\
-1\\
0\\
0
\end{bmatrix},
u_3=\begin{bmatrix}
-4m\epsilon\\
-4m\epsilon\\
3\epsilon\\
0\\
-2(m+1)\epsilon\\
0\\
8
\end{bmatrix},
u_4=\begin{bmatrix}
-2\\
-2\\
0\\
0\\
-1\\
0\\
0
\end{bmatrix}
\end{split}.
\end{equation}
Hence the general linearized solution near $P_{\mathbb{HP}^{m}}$ is in the form of 
\begin{equation}
\label{eqn_linearized_zeta}
P_{\mathbb{HP}^{m}}+s_1 e^\frac{2\eta}{3}u_1+s_2 e^\frac{2\eta}{3}u_2+s_3 e^\frac{2\eta}{3}u_3+s_4 e^\frac{2\eta}{3}u_4
\end{equation}
for some constants $s_i\in \mathbb{R}$. Scalar multiplication of $(s_1,s_2,s_3,s_4)$ is modded out by fixing $\sum_{i=1}^4s_i^2=1$. There is a 1 to 1 correspondence between germs of linearized solution \eqref{eqn_linearized_zeta} and $(s_1,s_2,s_3,s_4)\in\mathbb{S}^3$.
By Hartman--Grobman theorem, we can use $\zeta(s_1,s_2,s_3,s_4)$ to denote the actual solution that approaches \eqref{eqn_linearized_zeta} near $P_{\mathbb{HP}^{m}}$.  Moreover, from Theorem 4.5 in \cite{coddington_theory_1955}, there is some $\delta>0$ that
\begin{equation}
\label{eqn_linearized_zeta2}
\zeta(s_1,s_2,s_3,s_4) = P_{\mathbb{HP}^{m}}+s_1e^\frac{2\eta}{3}u_1+s_2e^\frac{2\eta}{3}u_2+s_3e^\frac{2\eta}{3}u_3+s_4e^\frac{2\eta}{3}u_4+O\left(  e^{\left(\frac{2}{3}+\delta\right)\eta} \right).
\end{equation}

It is clear that $P_{\mathbb{HP}^{m}}\in \mathcal{RS}_{\text{Einstein}}$. We set $s_1>0$ so that $Y_3$ is positive initially. The parameter $s_2$ squashes the $\mathbb{S}^3$-fiber. By \eqref{eqn_initial_condition_G/H} and the power series expansions of $a$ and $b$, we have 
\begin{equation}
\lim\limits_{\eta\to-\infty} \frac{X_2-X_1}{Y_3}=\lim\limits_{t\to 0} \frac{c^2}{b}\left(\frac{\dot{b}}{b}-\frac{\dot{a}}{a}\right)=\frac{c_0^2}{2}(\dddot{b}(0)-\dddot{a}(0)).
\end{equation}
In particular, if $s_2=0$, the corresponding component vanishes and the integral curve remains in $\mathcal{RS}_{\text{round}}$. We have
\begin{equation}
\label{eqn_gradient_at_PHP}
\nabla Q(P_{\mathbb{HP}^{m}})=
\begin{bmatrix}
\frac{2}{3}\\
\frac{4}{3}\\
0\\
-\frac{4}{9}\\
4\\
\frac{16m(m+2)}{3}\\
(n-1)\frac{\epsilon}{2}
\end{bmatrix},\quad \nabla H=\begin{bmatrix}
1\\
2\\
4m\\
0\\
0\\
0\\
0
\end{bmatrix},\quad \nabla W= \begin{bmatrix}
0\\
0\\
0\\
0\\
0\\
0\\
1
\end{bmatrix}.
\end{equation}
The eigenvector $u_3$ is the only one in \eqref{eqn_unstable_eigen_zeta} that is not tangent to $\mathcal{RS}_{\text{steady}}$. Thus, we set $s_3\geq 0$. Each $\zeta(s_1,s_2,0,s_4)$ stays in $\mathcal{RS}_{\text{steady}}$ and it represents a steady Ricci soliton. The eigenvector $u_4$ is the only one in \eqref{eqn_unstable_eigen_zeta} that is not tangent to $\mathcal{RS}_{\text{Einstein}}$. 
The functions $Q$ and $H-1$ are negative initially along $\zeta(s_1,s_2,s_3,s_4)$ with $s_4> 0$. Computations above show that there exists a continuous 3-parameter family of non-Einstein, non-shrinking Ricci solitons on a tubular neighborhood around $\mathbb{HP}^m$ in $\mathbb{HP}^{m+1}\backslash\{*\}$. The subfamily $\zeta(s_1,s_2,s_3,0)$ recovers the Einstein metrics on $\mathbb{HP}^{m+1}\backslash\{*\}$ in \cite{chi2021einstein}.

The other critical point $P_\bullet$ is also hyperbolic. The linearization at $P_\bullet$ is
\small
\begin{equation}
\begin{bmatrix}
-\frac{16m^2+20m+4}{(4m+3)^2} & \frac{4}{(4m+3)^2} & \frac{8m}{(4m+3)^2} & \frac{(8m+4)}{(4m+3)^2} & \frac{4}{4m+3} & \frac{8m}{4m+3} & \frac{(2m+1)\epsilon}{4m+3}\\
\frac{2}{(4m+3)^2} & -\frac{16m^2+20m+2}{(4m+3)^2} & \frac{8m}{(4m+3)^2} & -\frac{4}{(4m+3)^2} &  \frac{4}{4m+3} & \frac{8m}{4m+3} & \frac{(2m+1)\epsilon}{4m+3}\\
\frac{2}{(4m+3)^2} & \frac{4}{(4m+3)^2} & -\frac{16m^2+12m+6}{(4m+3)^2} & -\frac{4}{(4m+3)^2} &  \frac{4m+8}{4m+3} & \frac{4m-4}{4m+3} & \frac{(2m+1)\epsilon}{4m+3}\\
1 & -1 & 0 & 0 & 0 & 0 & 0\\
\frac{2}{(4m+3)^2} & \frac{(1-4m)}{(4m+3)^2} & \frac{8m}{(4m+3)^2} & 0 & 0 & 0 & -\frac{\epsilon}{8m+6}\\
\frac{2}{(4m+3)^2} & \frac{4m+7}{(4m+3)^2} & -\frac{6}{(4m+3)^2} & 0 & 0 & 0 & -\frac{\epsilon}{8m+6}\\
0 & 0 & 0 & 0 & 0 & 0 & \frac{2}{4m+3}
\end{bmatrix}.
\end{equation}
\normalsize
Each unstable eigenvector of $P_\bullet$ has $\frac{2}{n}$ as its eigenvalue. These unstable eigenvectors are
$$
v_1=\begin{bmatrix}
-4m\\
-4m\\
3\\
0\\
2m\\
-(2m+3)\\
0
\end{bmatrix}
v_2=\begin{bmatrix}
-(8m+4)\\
2\\
2\\
-(4m+3)^2\\
-1\\
-1\\
0
\end{bmatrix},\quad
v_3=\begin{bmatrix}
0\\
0\\
0\\
0\\
-\epsilon\\
-\epsilon\\
4
\end{bmatrix},\quad
v_4=\begin{bmatrix}
-2\\
-2\\
-2\\
0\\
-1\\
-1\\
0
\end{bmatrix}.
$$
Therefore, there exists a 3-parameter family of integral curves $\gamma(s_1,s_2,s_3,s_4)$ with $(s_1,s_2,s_3,s_4)\in\mathbb{S}^3$ that emanate from $P_{\bullet}$ such that
\begin{equation}
\label{eqn_linearized gamma negative einstein}
\gamma(s_1,s_2,s_3,s_4)=P_{\bullet}+s_1e^\frac{2\eta}{n}v_1+s_2e^\frac{2\eta}{n}v_2+s_3e^\frac{2\eta}{n}v_3+s_4e^\frac{2\eta}{n}v_4+O\left( e^{\left(\frac{2}{n}+\delta\right)\eta} \right).
\end{equation}

It is clear that $P_{\bullet}\in \mathcal{RS}_{\text{Einstein}}$. Furthermore, we have 
\begin{equation}
\label{eqn_gradient_at_H}
\nabla Q(P_\bullet)=
\begin{bmatrix}
\frac{2}{n}\\
\frac{4}{n}\\
\frac{8m}{n}\\
-\frac{8m+4}{n^2}\\
\frac{4(2m+3)(2m+1)}{n}\\
\frac{8m(2m+1)}{n}\\
(n-1)\frac{\epsilon}{2}
\end{bmatrix},
\end{equation}
while $\nabla H$ and $\nabla W$ are constant vector as in \eqref{eqn_gradient_at_PHP}. If $s_1=0$, the integral curve stays in $\mathcal{RS}_{\text{FS}}$. If $s_2=0$, the integral curve stays in $\mathcal{RS}_{\text{round}}$. We again set $s_3\geq 0$. Each $\zeta(s_1,s_2,0,s_4)$ stays in the invariant set $\mathcal{RS}_{\text{steady}}$ and it represents a steady Ricci soliton. As $v_4$ is not perpendicular to $\nabla Q (P_\bullet)$ or $\nabla H$, we set $s_4> 0$ so that $Q<0$ and $H<1$ initially along $\gamma(s_1,s_2,s_3,s_4)$. The subfamily $\gamma(s_1,s_2,s_3,0)$ recovers the Einstein metrics on $\mathbb{H}^{m+1}$ in \cite{chi2021einstein}.

To prove the global existence, we consider the following set
\begin{equation}
\label{eqn_setA}
\begin{split}
\mathcal{A}&:=\mathcal{RS}\cap\{Y_1\leq 1\}\cap\{X_1-X_2\leq 0\}\\
&\quad \cap\{Y_2-Y_3\geq 0\}\cap \left\{2(Y_2-Y_3)+X_3-X_2\geq 0\right\} \cap \{X_1,X_2,X_3\geq 0\}.
\end{split}
\end{equation}
The set $\mathcal{A}$ is extended from the compact invariant set $\mathcal{S}$ in \cite{chi2021einstein}. Proposition 4.2-4.3, and Lemma 4.4 in \cite{chi2021einstein} are easily generalized, and $\mathcal{A}$ is compact and invariant. For the sake of mathematical rigor, we present the complete proof below. Some readers may find certain arguments repetitive.
\begin{proposition}
\label{prop_Ais compact}
The set $\mathcal{A}$ is compact.
\end{proposition}
\begin{proof}
Since $Q\leq 0$ and $W\geq 0$ in $\mathcal{RS}$, we have 
$$
1\geq Q+1=G+R_s+(n-1)\frac{\epsilon}{2}W\geq R_s.
$$
by \eqref{eqn_curvature function}.
Since $Y_1\leq 1$ and $Y_2-Y_3\geq 0$, the above inequality implies $\frac{1}{6}\geq Y_2^2$. Thus, each $Y_i$ is bounded. The boundedness of each $X_i$ and $W$ follows from $Q\leq 0$.
\end{proof}

\begin{proposition}
\label{prop_technical_1}
If $2(Y_2-Y_3)+X_3-X_2=0$ on $\mathcal{RS}\cap \{Y_1\leq 1\}\cap \{X_3\geq 0\}\cap \{Y_2-Y_3\geq 0\}$, then
$$
1+4(m-1)Y_3-4Y_2+2Y_1^2\left(Y_2+Y_3\right)\geq 0
$$
holds.
\end{proposition}
\begin{proof}
The proof is almost identical to \cite[Proposition 4.3]{chi2021einstein}. The difference is we do not have the equality $Q=0$. The inequality $Q\leq 0$ does not bring extra difficulties to the argument, as shown in the following.

If $2(Y_2-Y_3)+X_3-X_2=0$, by \eqref{eqn_new_cohomogneity_one_conservation_law}, we have
\begin{equation}
\begin{split}
1&\geq Q+1\\
&=G+R_s+(n-1)\frac{\epsilon}{2}W\\
&\geq 2X_2^2+R_s\\
&= 2\left(X_3+2(Y_2-Y_3)\right)^2 +8Y_2^2-8mY_3^2 +4m(4m+8)Y_2Y_3-2Y_1^2Y_2^2-4mY_1^2Y_3^2.
\end{split}
\end{equation}
Since $X_3\geq 0$ and $Y_2-Y_3\geq 0$, we can drop $X_3$ above. The computation continues as
\begin{equation}
\label{eqn_key computation}
\begin{split}
1&\geq 8(Y_2-Y_3)^2+8Y_2^2-8mY_3^2 +4m(4m+8)Y_2Y_3-2Y_1^2Y_2^2-4mY_1^2Y_3^2\\
&=\left(16-2Y_1^2\right)Y_2^2+\left(8-8m-4mY_1^2\right)Y_3^2+(4m(4m+8)-16)Y_2Y_3\\
&\geq  \left(16-2Y_1^2\right)Y_2^2
\end{split}.
\end{equation}
The last inequality holds since $Y_2\geq Y_3$ and $Y_1^2\leq 1$.

By \eqref{eqn_key computation}, we have 
$\frac{1}{16-2Y_1^2}\geq Y_2^2$ if the equality $2(Y_2-Y_3)+X_3-X_2=0$ holds. Furthermore, we have
$$
\left(\frac{1}{4-2Y_1^2}\right)^2\geq \frac{1}{16-2Y_1^2}\geq Y_2^2.
$$
Therefore,
$$
\frac{1}{4}+\frac{1}{2}Y_1^2Y_2\geq Y_2,
$$
and the inequality
$$1+4(m-1)Y_3-4Y_2+2Y_1^2\left(Y_2+Y_3\right)\geq 0
$$
holds.
\end{proof}

\begin{lemma}
\label{lem_invariant_setA}
The set
\begin{equation}
\begin{split}
\mathcal{A}&:=\mathcal{RS}\cap\{Y_1\leq 1\}\cap\{X_1-X_2\leq 0\}\\
&\quad \cap\{Y_2-Y_3\geq 0\}\cap \left\{2(Y_2-Y_3)+X_3-X_2\geq 0\right\} \cap \{X_1,X_2,X_3\geq 0\}
\end{split}
\end{equation}
is invariant.
\end{lemma}

\begin{proof}
Since
\begin{equation}
\begin{split}
Y_1'&=Y_1(X_1-X_2),\\
(X_1-X_2)'&=(X_1-X_2)\left(G-\frac{\epsilon}{2}W-1\right)+(Y_1^2-1)(4Y_2^2+4mY_3^2),
\end{split}
\end{equation}
the vector field $V$ on $\mathcal{A}\cap \{Y_1-1\}$ and $\mathcal{A}\cap \{X_1-X_2=0\}$ points inward.

In order to check the face $\mathcal{A}\cap\{X_i=0\}$, it suffices to show that $R_i\geq 0$ in $\mathcal{A}$. It is obvious that $R_1\geq 0$. From defining inequalities $Y_1\leq 1$ and $Y_2-Y_3\geq 0$, we have
\begin{equation}
\label{eqn_Ricci_curvature_R2_R3}
\begin{split}
R_2&=4Y_2^2-2Y_1^2Y_2^2+4mY_3^2\geq 2Y_2^2+4mY_3^2\geq 0 \\
R_3&=(4m+8)Y_2Y_3-2Y_1^2Y_3^2-4Y_3^2\geq (4m+2)Y_3^2\geq 0
\end{split}.
\end{equation}
Therefore, the vector field $V$ on each face $\mathcal{A}\cap \{X_i=0\}$ points inward. Furthermore, from $H\leq 1$, we know that 
\begin{equation}
\label{eqn_upper_bound_for_X}
X_1\leq 1,\quad X_2\leq \frac{1}{2},\quad X_3\leq \frac{1}{4m}
\end{equation}
in $\mathcal{A}$.

For the face $\mathcal{A}\cap \left\{Y_2-Y_3= 0\right\}$, we have 
\begin{equation}
\label{eqn_key_computation0.5}
\begin{split}
(Y_2-Y_3)'&=(Y_2-Y_3)\left(G-\frac{\epsilon}{2}W-X_2\right)+2Y_3(X_3-X_2)\\
&\geq (Y_2-Y_3)\left(G-\frac{\epsilon}{2}W-X_2\right)+4Y_3(Y_3-Y_2).
\end{split}
\end{equation}
Thus, the vector field $V$ on each face $\mathcal{A}\cap \{Y_2-Y_3=0\}$ points inward.

For the face $\mathcal{A}\cap \left\{2(Y_2-Y_3)+X_3-X_2=0\right\}$, we have
\begin{equation}
\label{eqn_key_computation1}
\begin{split}
&\left(2(Y_2-Y_3)+X_3-X_2\right)'\\
&= \left(2(Y_2-Y_3)+X_3-X_2\right)\left(G-\frac{\epsilon}{2}W-1+4Y_3\right)\\
&\quad +(Y_2-Y_3)\left(2-2X_2+4(m-1)Y_3-4Y_2+2Y_1^2(Y_2+Y_3)\right)\\
&\geq (Y_2-Y_3)\left(1+4(m-1)Y_3-4Y_2+2Y_1^2(Y_2+Y_3)\right). \\
&\quad \text{by $2(Y_2-Y_3)+X_3-X_2=0$ and \eqref{eqn_upper_bound_for_X}}
\end{split}
\end{equation}
By Proposition \ref{prop_technical_1}, the vector field restricted on $\mathcal{A}\cap \left\{2(Y_2-Y_3)+X_3-X_2=0\right\}$ points inward. The proof is complete.
\end{proof}

We are ready to prove the global existence of Ricci solitons represented by the integral curves $\zeta(s_1,s_2,s_3,s_4)$ and  $\gamma(s_1,s_2,s_3,s_4)$.

\begin{lemma}
\label{lem_long existing_zeta_gamma}
Integral curves in $\{\zeta(s_1,s_2,s_3,s_4)\mid \sum_{i=1}^4s_i^2=1, s_1,s_4>0, s_2,s_3\geq 0\}$ and $\{\gamma(s_1,s_2,s_3,s_4)\mid \sum_{i=1}^4s_i^2=1, s_1,s_2,s_3\geq 0,s_4>0\}$ are defined on $\mathbb{R}$. Metrics represented by these integral curves are complete.
\end{lemma}
\begin{proof}
Functions $Y_2-Y_3$, $2(Y_2-Y_3)+X_3-X_2$, $X_1$, and $X_2$ are positive at $P_{\mathbb{HP}^m}$. From \eqref{eqn_linearized_zeta}, it is also clear that functions $1-Y_1$, $X_1-X_2$ and $X_3$ are non-negative for an integral curve $\zeta(s_1,s_2,s_3,s_4)$ with $s_1,s_4>0$ and $s_2,s_3\geq 0$. In particular, if $s_2=0$, the integral curve lies on the invariant set $\mathcal{RS}_{\text{round}}\cap\mathcal{A}$. If $s_2>0$, the integral curve is in the interior of $\mathcal{A}$ once it leaves $P_{\mathbb{HP}^m}$.

From the fact that $\mathcal{RS}_{\text{round}}$ and $\mathcal{RS}_{\text{FS}}$ are invariant and computation \eqref{eqn_key_computation1}, it is impossible for $\zeta(s_1,s_2,s_3,s_4)$ to escape $\mathcal{A}$ non-traversally through the faces 
$$\mathcal{A}\cap \{Y_1=1\},\quad \mathcal{A}\cap \{X_1-X_2=0\},\quad \mathcal{A}\cap \{Y_2-Y_3=0\},\quad \mathcal{A}\cap \left\{2(Y_2-Y_3)+X_3-X_2=0\right\}.$$
As $R_i\geq 0$ in $\mathcal{A}$, it is also clear that $\zeta(s_1,s_2,s_3,s_4)$ does not escape $\mathcal{A}$ non-traversally through each $\mathcal{A}\cap \{X_i=0\}$. Hence, each $\zeta(s_1,s_2,s_3,s_4)$ stays in the compact set $\mathcal{A}$ and it is defined on $\mathbb{R}$ by the escape lemma. By Proposition \ref{prop_etainfty is tinfty}, metrics represented by these integral curves are complete. 

With the similar argument, each $\gamma(s_1,s_2,s_3,s_4)$ with $s_1,s_2,s_3,s_4\geq 0$ is defined on $\mathbb{R}$. If $s_1=0$, the integral curve stays in $\mathcal{RS}_{\text{FS}}\cap\mathcal{A}$. If $s_2=0$, the integral curve stays in $\mathcal{RS}_{\text{round}}\cap\mathcal{A}$. The metrics represented by these integral curves are complete.
\end{proof}

\begin{remark}
\label{remark_KR solitons}
By straightforward computations, the integral curve $\gamma(0,s_2,s_3,s_4)$ is in $\mathcal{RS}_{\text{KE}}$ if $s_4=(n+3)s_2+\epsilon s_3$. Note that the integral curve $\gamma(0,\sigma,0,(n+3)\sigma), \sigma=\frac{1}{\sqrt{n^2+6n+10}}$ represents the flat Euclidean metric on $\mathbb{C}^{2m+2}$ with a non-trivial potential function. We thus recover the 1-parameter of K\"ahler--Ricci solitons in \cite{cao1997limits}. It is noteworthy that the 1-parameter family $\gamma(0,s_2,s_3,(n+3)s_2+\epsilon s_3)$ with $s_2\geq 0$ only recovers those in \cite{cao1997limits} with non-negative sectional curvature, where the principal orbit is squashed in a way such that $a\leq b$.
\end{remark}

\begin{remark}
\label{rem_positive Ricci}
It is noteworthy that as $R_i\geq 0$ in $\mathcal{A}$, the Ricci curvatures of $\mathsf{G}/\mathsf{K}$ in each cohomogeneity one soliton in Lemma \ref{lem_long existing_zeta_gamma} are positive. The scalar curvature $r_s$ of the principal orbit is positive.
\end{remark}

\section{Asymptotics}
We investigate asymptotics for non-Einstein Ricci solitons in Lemma \ref{lem_long existing_zeta_gamma}. We first establish some general results. Since each integral curve in Lemma \ref{lem_long existing_zeta_gamma} enters $\mathcal{A}$, the defining inequalities \eqref{eqn_setA} are used in the following.

\begin{proposition}
\label{prop_points with Y1 nonzero}
Integral curves in Lemma \ref{lem_long existing_zeta_gamma} with $s_4>0$ converge to $(0,0,0,\mu,0,0,0)$ for some $\mu\in [0,1]$. 
\end{proposition}
\begin{proof}
For $\zeta(s_1,s_2,s_3,s_4)$ and $\gamma(s_1,s_2,s_3,s_4)$ in Lemma \ref{lem_long existing_zeta_gamma}, the function $Y_1$ is a constant if $s_2=0$, monotonically decreasing to some $\mu\geq 0$ if $s_2>0$. 

Consider $\epsilon=0$. By \eqref{eqn_characteriaztion_of_non_trivial_soliton}, the function $Q$ monotonically decreases to some negative number. Hence, the $\omega$-limit set of the integral curve is contained in $\mathcal{RS}_{\text{steady}}\cap \{X_i=0\}$. As the $\omega$-limit set is invariant, we further conclude that the $\omega$-limit set must be contained in $\mathcal{RS}_{\text{steady}}\cap \{X_i,Y_2,Y_3=0\}$. 

For $\epsilon=1$, there exists a large enough $t_*$ such that the inequality $-\dot{f}\geq \alpha t+\beta$ holds for all $t>t_*$, where $\alpha$ and $\beta$ are positive constants; see \cite[Proposition 1.13 and Proposition 1.18]{buzano_non-kahler_2015}. Each principal curvature is non-negative by $X_i\geq 0$. By the definition of $W$, we have
$$
\lim\limits_{\eta\to\infty} W=\lim\limits_{t\to\infty} \frac{1}{(\trace(L)-\dot{f})^2}\leq \lim\limits_{t\to\infty} \frac{1}{(\alpha t+\beta)^2}=0.
$$ 
Hence, the function $W$ converges to zero. By $\left(\frac{W^2}{Y_2^2}\right)'=2\frac{W^2}{Y_2^2}X_2\geq 0$, the function $\frac{W^2}{Y_2^2}$ is non-decreasing along the integral curves. Thus, it is necessary that $\lim\limits_{\eta\to\infty}Y_2=0$. Since $Y_2\geq Y_3$, the function $Y_3$ also converges to zero. The $\omega$-limit set is contained in $\mathcal{RS}\cap \{W,Y_2,Y_3=0\}$. As the $\omega$-limit set is invariant, and we have $X_i'=X_i(G-1)$, the integral curve converges to $(0,0,0,\mu,0,0,0)$ or $(1,0,0,\mu,0,0,0)$. By $X_1-X_2\leq 0$, the integral curve converges to $(0,0,0,\mu,0,0,0)$.
\end{proof}

We obtain the following corollary from Proposition \ref{prop_points with Y1 nonzero}.
\begin{corollary}
\label{cor_Q to -1}
The function $Q$ converges to $-1$ along each integral curve with $s_4>0$ in Lemma \ref{lem_long existing_zeta_gamma}.
\end{corollary}

Proposition \ref{prop_points with Y1 nonzero} indicates that it is insufficient to know where $\zeta$ and $\gamma$ converge to understand the asymptotic behavior of a non-Einstein Ricci soliton. As most of the variables vanish at the point of convergence, it is natural to investigate \emph{how} the integral curves converge. With the existence of $\mu$ by Proposition \ref{prop_points with Y1 nonzero}, the next key step is to prove the existence of 
$$\nu=\lim\limits_{t\to\infty} \frac{b}{c}=\lim\limits_{\eta\to\infty} \sqrt{\frac{Y_3}{Y_2}}.$$ 
Before proceeding further, we introduce the following proposition, a slight generalization to \cite[Lemma 3.8]{dancer_new_2009}.

\begin{proposition}
\label{prop_model_ode}
Suppose a positive function $P$ satisfies the differential equation 
\begin{equation}
P'=XP+U,
\end{equation}
where $X$ and $U$ are functions on $\mathbb{R}$. If the limits $\lim\limits_{\eta\to \infty} X=-1$ and $\lim\limits_{\eta\to \infty}U= u\geq 0$ exist, the limit $\lim\limits_{\eta\to\infty} P=u$ also exists.
\end{proposition}
\begin{proof}
If $P$ converges to some $p$, its derivative converges to $-p+u$. Hence, it is necessary that $p=u$. It suffices to prove the existence of $\lim\limits_{\eta\to\infty} P$.

If $u=0$, then for any $\delta \in (0,1)$,  there exists an $\eta_*$ large enough so that $\eta>\eta_*$ implies
\begin{equation}
|U|<\delta ,\quad X<-(1-\delta ).
\end{equation}
Since $P$ is a positive function, we also have 
\begin{equation}
\begin{split}
P'&=XP+U\leq -(1-\delta )P+\delta 
\end{split}.
\end{equation}
If $P(\eta_0)\geq \frac{2\delta }{1-\delta }$ for some $\eta_0>\eta_*$, then 
\begin{equation}
\label{eqn_negative P'}
P'(\eta_0)\leq -(1-\delta )\frac{2\delta }{1-\delta }+\delta =-\delta <0.
\end{equation}
If $P\geq \frac{2\delta }{1-\delta }$ for all $\eta>\eta_*$, the computation above implies that $P$ monotonically decreases to some number that is at least $\frac{2\delta}{1-\delta }$. As discussed above, if $\lim\limits_{\eta\to\infty} P$ exists, the limit has to be $u=0$. We thus reach a contradiction. Therefore, the function $P$ enters the interval $\left(0,\frac{2\delta }{1-\delta}\right)$ for some $\eta_{**}>\eta_0>\eta_*$. Inequality \eqref{eqn_negative P'} implies that $P$ stays in interval $\left(0,\frac{2\delta }{1-\delta}\right)$ for any $\eta>\eta_{**}$. By the arbitrariness of $\delta\in (0,1)$, we conclude that $\lim\limits_{\eta\to \infty} P=0$.

Assume $u\neq 0$ in the following.
For any $\delta \in (0,1)$, there exists an $\eta_*$ large enough so that $\eta>\eta_*$ implies
\begin{equation}
|U-u|<\delta ^2,\quad X<-(1-\delta ^2).
\end{equation}
For $\eta>\eta_*$, we also have 
\begin{equation}
\begin{split}
P'&=XP+U\leq -(1-\delta ^2)P+U
\end{split}.
\end{equation}
Consider $\delta $ small enough so that $-\delta  u+\delta ^2<0$. If $P(\eta_0)\geq \frac{u}{1-\delta }$ for some $\eta_0>\eta_*$, then 
$$P'(\eta_0)\leq -(1-\delta ^2)\frac{u}{1-\delta }+U=-\delta  u+(U-u)<-\delta  u+\delta ^2<0;$$
If $P(\eta_0)\leq \frac{u}{1+\delta }$ for some $\eta_0>\eta_*$, then 
$$P'(\eta_0)\geq -(1-\delta ^2)\frac{u}{1+\delta }+U=\delta  u+(U-u)>\delta  u-\delta ^2>0.$$ 
If $\frac{u}{1+\delta }<P(\eta_0)<\frac{u}{1-\delta}$ for some $\eta_0>\eta_*$, the function $P$ stays in the set for all $\eta>\eta_0$. Then, the rest of the argument is the same as the one in \cite[Lemma 3.8]{dancer_new_2009}. For any $\delta \in (0,1)$ such that $-\delta  u+\delta ^2<0$, we must have $\frac{u}{1+\delta }<P<\frac{u}{1-\delta }
$ for all large enough $\eta$. By the arbitrariness of $\delta$, the proof is complete.
\end{proof}

\subsection{Steady Ricci-solitons}
For the steady case, proving the existence of $\nu$ depends on the value of $\mu$. If $\mu \neq 0$, we can apply \cite[Proposition 2.22]{wink_cohomogeneity_2017}. This leads us to the following proposition.
\begin{proposition}
\label{prop_b/c limit exists}
Suppose $\lim\limits_{t\to \infty}\frac{a}{b}=\mu>0$ along $\zeta(s_1,s_2,0,s_4)$ or $\gamma(s_1,s_2,0,s_4)$ in Lemma \ref{lem_long existing_zeta_gamma}. Then $\lim\limits_{t\to \infty}\frac{b}{c}=\nu$ for some $\nu> 0$.
\end{proposition}
\begin{proof}
Let $v:=ab^2c^{4m}$ and consider the B\"ohm's functional
\begin{equation}
\label{eqn_bohm_functional}
B:=v^\frac{2}{n}\left(r_s+\trace(L_0^2)\right),
\end{equation}
where $L_0$ is the traceless part of $L$. From \cite[Proposition 2.17]{dancer_cohomogeneity_2013}, we have
\begin{equation}
\label{eqn_derivative of bohm functional}
\dot{B}=-2v^{\frac{2}{n}}\trace\left(L_0^2\right) \left(\frac{n-1}{n}\trace{(L)}-\dot{f}\right).
\end{equation}
By Proposition \ref{prop_points with Y1 nonzero}, each $\frac{X_i}{\sqrt{-Q}}$ converges to zero as $\eta\to \infty$. Hence, each principal curvature converges to zero as $t\to \infty$. By Corollary \ref{cor_Q to -1}, we have $\lim\limits_{t\to\infty} \dot{f}=-\sqrt{-C}$. Therefore, the function $B$ is eventually non-increasing. By the Cauchy--Schwarz inequality and Remark \ref{rem_positive Ricci}, we have $B\geq v^\frac{2}{n}r_s\geq 0$. Thus, the function converges. Since
\begin{equation}
\begin{split}
v^\frac{2}{n}r_s &= 8\left(\frac{a}{b}\right)^\frac{2}{n}\left(\frac{c}{b}\right)^\frac{8m}{n}+4m(4m+8)\left(\frac{a}{b}\right)^\frac{2}{n}\left(\frac{b}{c}\right)^\frac{6}{n}\\
&\quad -2\left(\frac{a}{b}\right)^{2+\frac{2}{n}}\left(\frac{c}{b}\right)^{\frac{8m}{n}}-4m\left(\frac{a}{b}\right)^{2+\frac{2}{n}}\left(\frac{b}{c}\right)^{2+\frac{6}{n}}-8m\left(\frac{a}{b}\right)^\frac{2}{n}\left(\frac{b}{c}\right)^{2+\frac{6}{n}},
\end{split}
\end{equation}
the ratio $\frac{b}{c}$ is eventually bounded away from zero.

Define $$\tilde{a}=v^\frac{2}{n}\frac{\dot{a}}{a},\quad \tilde{b}=v^\frac{2}{n}\frac{\dot{b}}{b},\quad \tilde{c}=v^\frac{2}{n}\frac{\dot{c}}{c}.$$
Straightforward computations show that 
\begin{equation}
\label{eqn_derivativeoftilde_abc}
\begin{split}
\dot{\tilde{a}}&=-\left(\frac{n-2}{n}\trace(L)-\dot{f}\right)\tilde{a}+2\left(\frac{a}{b}\right)^{2+\frac{2}{n}}\left(\frac{c}{b}\right)^{\frac{8m}{n}}+4m\left(\frac{a}{b}\right)^{2+\frac{2}{n}}\left(\frac{b}{c}\right)^{2+\frac{6}{n}},\\
\dot{\tilde{b}}&=-\left(\frac{n-2}{n}\trace(L)-\dot{f}\right)\tilde{b}+4\left(\frac{a}{b}\right)^\frac{2}{n}\left(\frac{c}{b}\right)^\frac{8m}{n}-2\left(\frac{a}{b}\right)^{2+\frac{2}{n}}\left(\frac{c}{b}\right)^{\frac{8m}{n}}+4m\left(\frac{a}{b}\right)^\frac{2}{n}\left(\frac{b}{c}\right)^{2+\frac{6}{n}},\\
\dot{\tilde{c}}&=-\left(\frac{n-2}{n}\trace(L)-\dot{f}\right)\tilde{c}+(4m+8)\left(\frac{a}{b}\right)^\frac{2}{n}\left(\frac{b}{c}\right)^\frac{6}{n}-2\left(\frac{a}{b}\right)^{2+\frac{2}{n}}\left(\frac{b}{c}\right)^{2+\frac{6}{n}}-4\left(\frac{a}{b}\right)^\frac{2}{n}\left(\frac{b}{c}\right)^{2+\frac{6}{n}}.
\end{split}
\end{equation}
Zeroth order terms in \eqref{eqn_derivativeoftilde_abc} are bounded. As  $\lim\limits_{t\to\infty}\left(\frac{n-2}{n}\trace(L)-\dot{f}\right)=\sqrt{-C}>0$, functions $\tilde{a}$, $\tilde{b}$ and $\tilde{c}$ must be bounded. Therefore, the function $v^{\frac{2}{n}}\trace(L_0^2)$ converges to zero. We have
\begin{equation}
\label{eqn_limit of B}
\begin{split}
\lim\limits_{t\to\infty} B=\lim\limits_{t\to\infty} v^\frac{2}{n} r_s&=\lim\limits_{t\to\infty}\left(8\mu^\frac{2}{n}\left(\frac{c}{b}\right)^\frac{8m}{n}+4m(4m+8)\mu^\frac{2}{n}\left(\frac{b}{c}\right)^\frac{6}{n}\right.\\
&\quad \left.-2\mu^{2+\frac{2}{n}}\left(\frac{c}{b}\right)^{\frac{8m}{n}}-4m\mu^{2+\frac{2}{n}}\left(\frac{b}{c}\right)^{2+\frac{6}{n}}-8m\mu^\frac{2}{n}\left(\frac{b}{c}\right)^{2+\frac{6}{n}}\right).
\end{split}
\end{equation}
As $\lim\limits_{t\to\infty} B$ exists and $\mu\neq 0$, the limit $\lim\limits_{t\to\infty} \frac{b}{c}$ also exists.
\end{proof}

If $\mu=0$, we have $\lim\limits_{\eta\to\infty}B=0$ by \eqref{eqn_limit of B} and the existence of $\nu$ is no longer guaranteed. In this case, we employ an alternative method to establish the existence of $\nu$. Define 
$$
\tilde{R}_2=4Y_2^2+4m Y_3^2,\quad \tilde{R}_3=(4m+8)Y_2Y_3-4Y_3^2.
$$
We have the following proposition.
\begin{proposition}
\label{prop_fixed ratio}
If $\mu=0$, we have the limit
$\lim\limits_{\eta\to\infty}\frac{X_2}{X_3}\frac{\tilde{R}_3}{\tilde{R}_2}=1$.
\end{proposition}
\begin{proof}
We have 
\begin{equation}
\begin{split}
&\left(\frac{X_2}{\tilde{R}_2}\right)'=\frac{X_2}{\tilde{R}_2}\left(-G-1+\frac{8Y_2^2}{\tilde{R}_2}X_2-\frac{8mY_3^2}{\tilde{R}_2}(X_2-2X_3)\right)+1-2\frac{Y_1^2Y_2^2}{\tilde{R}_2},\\
&\left(\frac{X_3}{\tilde{R}_3}\right)'=\frac{X_3}{\tilde{R}_3}\left(-G-1+\frac{2(4m+8)Y_2Y_3}{\tilde{R}_3}X_3+\frac{8Y_3^2}{\tilde{R}_3}(X_2-2X_3)\right)+1-2Y_1^2\frac{Y_3^2}{\tilde{R}_3}.
\end{split}
\end{equation}
Since $\frac{Y_3}{Y_2}\in[0,1]$, each one of $\frac{Y_2^2}{\tilde{R}_2}$, $\frac{Y_3^2}{\tilde{R}_3}$ and $\frac{Y_2Y_3}{\tilde{R}_3}$ is bounded. By Proposition \ref{prop_points with Y1 nonzero}, the coefficient for each $\frac{X_i}{\tilde{R}_i}$ converges to $-1$. By the assumption of $\mu=0$ and Proposition \ref{prop_model_ode}, each $\frac{X_i}{\tilde{R}_i}$ converge to $1$. The proof is complete by taking the ratio.
\end{proof}

If $X_2-X_3$ eventually has a sign, by $\left(\frac{Y_3}{Y_2}\right)'=2\frac{Y_3}{Y_2}(X_2-X_3)$, the function $\frac{Y_3}{Y_2}$ is eventually decreasing or eventually increasing. Hence, the limit $\nu$ exists by the compactness of $\mathcal{A}$. We focus on the case where $X_2-X_3$ changes sign infinitely many times in the following.

\begin{proposition}
\label{prop_infinite times analysis 1}
If $\mu=0$ and the function $X_2-X_3$ changes sign infinitely many times, the integral curve eventually stays in the set
$$\Psi:=\mathcal{RS}\cap \left\{X_2-X_3-Y_2+(m+1)Y_3\leq 0\right\}\cap \left\{Y_2-(m+1)Y_3\geq 0\right\}.$$
\end{proposition}
\begin{proof}
Since all $X_i$ and $Y_i$ converge to zero, we first choose an $\eta_*$ such that $K:=1-X_2-4Y_2-(2m-2)Y_3>0$ for $\eta>\eta_{*}$. 

We claim that there is an $\eta_{**}\geq \eta_*$ such that $(X_2-X_3)(\eta_{**})<0$ and $\left(Y_2-(m+1)Y_3\right)(\eta_{**})>0.$ Suppose such an $\eta_{**}$ does not exist, the inequality $Y_2-(m+1)Y_3\leq 0$ holds whenever $X_2-X_3<0$. By
\begin{equation}
\begin{split}
(X_2-X_3)'&=(X_2-X_3)(G-1)+(Y_2-Y_3)\left(4Y_2-(4m+4)Y_3-2Y_1^2(Y_2+Y_3)\right)\\
&\leq (X_2-X_3)(G-1),
\end{split}
\end{equation}
we have $X_2-X_3\leq 0$ for all $\eta\geq \eta_{*}$, a contradiction to our assumption.

With such an $\eta_{**}$, the integral curve enters the set $\Psi$.
Next, we show that the integral curve stays in the set. Suppose otherwise, the integral curve escapes the set through one of the two faces at some $\eta_{***}>\eta_{**}$. Suppose $(X_2-X_3-Y_2+(m+1)Y_3)(\eta_{***})=0$. At that point, we have 
\begin{equation}
\label{eqn_derivative at eta**}
\begin{split}
&\left(X_2-X_3-Y_2+(m+1)Y_3\right)'(\eta_{***})\\
&=(X_2-X_3-Y_2+(m+1)Y_3)(G-1+2(m+1)Y_3)-2Y_1^2(Y_2^2-Y_3^2)\\
&\quad -(Y_2-(m+1)Y_3)K\\
&= -2Y_1^2(Y_2^2-Y_3^2)-(Y_2-(m+1)Y_3)K\\
&\leq -(Y_2-(m+1)Y_3)K
\end{split}
\end{equation}
Since $K(\eta_{***})>0$, the derivative above is non-positive. Non-traversal crossing is impossible as it requires $Y_1^2(Y_2^2-Y_3^2)$ and $Y_2-(m+1)Y_3$ to vanish simultaneously. Thus, the integral curve does not escape through the face $\mathcal{RS}\cap \left\{X_2-X_3-Y_2+(m+1)Y_3= 0\right\}$.

Suppose $(Y_2-(m+1)Y_3)(\eta_{***})=0$. At that point, we have 
\begin{equation}
\begin{split}
\left(\frac{Y_3}{Y_2}\right)'&=2\left(\frac{Y_3}{Y_2}\right)(X_2-X_3)\leq 2\left(\frac{Y_3}{Y_2}\right)\left(Y_2-(m+1)Y_3\right)=0
\end{split}
\end{equation}
If the crossing is non-traversal, both $Y_2-(m+1)Y_3$ and $X_2-X_3$ vanish at the crossing point. We reach a contradiction that
\begin{equation}
\begin{split}
\left(\frac{Y_3}{Y_2}\right)''(\eta_{***})&=2\left(\frac{Y_3}{Y_2}\right)(Y_2-Y_3)\left(4Y_2-(4m+4)Y_3-2Y_1^2(Y_2+Y_3)\right)\\
&=2\left(\frac{Y_3}{Y_2}\right)(Y_2-Y_3)\left(-2Y_1^2(Y_2+Y_3)\right)\\
&<0.
\end{split}
\end{equation}
Hence, the integral curve does not escape the set through the face $\mathcal{RS}\cap \left\{Y_2-(m+1)Y_3=0\right\}$.
\end{proof}

We are ready to prove the existence of $\nu$.

\begin{proposition}
\label{prop_nu exists even mu=0}
If $\mu=0$ and $X_2-X_3$ changes sign infinitely many times, we have the limit $\nu^2=\frac{1}{m+1}$.
\end{proposition}
\begin{proof}
Consider any sequence $\{a_k\}_{k=1}^\infty$ with $\lim\limits_{k\to\infty}a_k=\infty$ such that $\left(\frac{Y_3}{Y_2}\right)(a_k)$ is a local extremum for each $k$. By Proposition \ref{prop_infinite times analysis 1}, the integral curve eventually enters the set $\Psi$. We can assume $\left(\frac{Y_3}{Y_2}\right)(a_k)<\frac{1}{m+1}$ without loss of generality. By Proposition \ref{prop_fixed ratio}, we have  
$$
\lim\limits_{k\to\infty} \left(\frac{\tilde{R}_3}{\tilde{R}_2}\right)(a_k)=\lim\limits_{k\to\infty}\left(\frac{(4m+8)\frac{Y_3}{Y_2}-4\frac{Y_3^2}{Y_2^2}}{4 +4m \frac{Y_3^2}{Y_2^2}}\right)(a_k)=1.
$$
Since the rational function $y=\frac{(4m+8)x-4x^2}{4+4mx^2}$ is increasing on $\left[0,\frac{1}{m+1}\right]$, the computation above indicates that $\left\{\left(\frac{Y_3}{Y_2}\right)(a_k)\right\}_{k=1}^\infty$ as a sequence converges to $\frac{1}{m+1}$.

We proceed to show that $\frac{Y_3}{Y_2}$ converges to $\frac{1}{m+1}$ as a function. Suppose otherwise, there exists a $\delta>0$ and a sequence $\{c_k\}_{k=1}^\infty$ with $\lim\limits_{k\to\infty}c_k=\infty$ such that $\left(\frac{Y_3}{Y_2}\right)(c_k)\leq \frac{1}{m+1}-\delta$ for each $k$. Without loss of generality, each $c_k$ is not a local extremum. Let $\{b_k\}_{k=1}^\infty$ be the sequence where $b_k$ is the local minimum for $\frac{Y_3}{Y_2}$ that is closest to $c_k$. We have 
$$\left(\frac{Y_3}{Y_2}\right)(b_k)\leq \left(\frac{Y_3}{Y_2}\right)(c_k)\leq \frac{1}{m+1}-\delta,$$ which is a contradiction.
%
%
%
%
\end{proof}


By Proposition \ref{prop_points with Y1 nonzero}, Proposition \ref{prop_b/c limit exists} and Proposition \ref{prop_nu exists even mu=0}, we know that $(\mu^2, \nu^2)$ exists.  We are ready to describe all the possibilities of asymptotics for a steady Ricci soliton.
\begin{lemma}
\label{lem_possible asymptotics}
Possible values for $(\mu^2,\nu^2)$ are the following.
$$
(1,1),\quad (0,1),\quad \left(1,\frac{1}{2m+3}\right),\quad \left(0,\frac{1}{m+1}\right).
$$
\end{lemma}
\begin{proof}
Consider 
\begin{equation}
\begin{split}
\left(\frac{X_1Y_1^2}{Y_2^2}\right)'&=-(G+1-2X_1)\frac{X_1Y_1^2}{Y_2^2}+2 Y_1^4+4 m Y_1^4\frac{Y_3^2}{Y_2^2},\\
\left(\frac{X_2}{Y_2^2}\right)'&=-(G+1-2X_2)\frac{X_2}{Y_2^2}+4- 2 Y_1^2+4 m \frac{Y_3^2}{Y_2^2},\\
\left(\frac{X_3}{Y_2Y_3}\right)'&=-(G+1-2X_3)\frac{X_3}{Y_2Y_3}+(4m+8)-2 Y_1^2\frac{Y_3}{Y_2}-4 \frac{Y_3}{Y_2}.
\end{split}
\end{equation}
By Proposition \ref{prop_model_ode}, we have
\begin{equation}
\label{eqn_limit of X2/Y2^2}
\lim\limits_{\eta\to\infty}\frac{X_1Y_1^2}{Y_2^2}=2 \mu^4+4m \mu^4\nu^4,\quad \lim\limits_{\eta\to\infty}\frac{X_2}{Y_2^2}=4- 2 \mu^2+4m \nu^4,\quad  \lim\limits_{\eta\to\infty}\frac{X_3}{Y_2Y_3}=(4m+8)-2 \mu^2\nu^2-4\nu^2.
\end{equation}
Hence,
\begin{equation}
\begin{split}
\lim\limits_{t\to\infty}a\dot{a}&=\lim\limits_{\eta\to\infty}\frac{Q}{C}\frac{X_1Y_1^2}{Y_2^2}=\frac{1}{-C}\left(2 \mu^4+4 m \mu^4\nu^4\right),\\
\lim\limits_{t\to\infty}b\dot{b}&=\lim\limits_{\eta\to\infty}\frac{Q}{C}\frac{X_1Y_1^2}{Y_2^2}=\frac{1}{-C}\left(4-2 \mu^2+4 m \nu^4\right),\\
\lim\limits_{t\to\infty}c\dot{c}&=\lim\limits_{\eta\to\infty}\frac{Q}{C}\frac{X_1Y_1^2}{Y_2^2}=\frac{1}{-C}\left((4m+8)-2\mu^2\nu^2-4 \nu^2 \right).
\end{split}
\end{equation}
Suppose $\mu\neq 0$, then each limit above is positive. As $t\to \infty$ we have
\begin{equation}
\label{eqn_paraboloid asymtptotics}
\begin{split}
a^2&\sim \frac{2}{-C}\left(2 \mu^4+ 4m \mu^4\nu^4\right)t,\\
b^2&\sim  \frac{2}{-C}\left(4-2 \mu^2+4m \nu^4\right)t,\\
c^2&\sim   \frac{2}{-C}\left((4m+8)-2\mu^2\nu^2-4\nu^2\right)t.
\end{split}
\end{equation}
Hence, each principal curvature grows as fast as $\frac{1}{2t}$ as $t\to\infty$. 
Since $\frac{\dot{a}}{a}\sim\frac{\dot{b}}{b}\sim\frac{\dot{c}}{c}\sim\frac{1}{2t},$ it follows that $\lim\limits_{\eta\to\infty}\frac{X_1}{X_2}=\lim\limits_{\eta\to\infty}\frac{X_2}{X_3}=1$. By the L'H\^opital's rule, we have 

\begin{equation}
\begin{split}
\mu^2&=\lim\limits_{t\to\infty}\frac{a^2}{b^2}=\lim\limits_{t\to\infty}\frac{a\dot{a}}{b\dot{b}}
=\lim\limits_{\eta\to\infty}\frac{\frac{X_1Y_1^2}{Y_2^2}}{\frac{X_2}{Y_2^2}}=\frac{2 \mu^4+ 4 m \mu^4\nu^4}{4-2 \mu^2+4 m \nu^4},\\
\nu^2&=\lim\limits_{t\to\infty}\frac{b^2}{c^2}=\lim\limits_{t\to\infty}\frac{b\dot{b}}{c\dot{c}}
=\lim\limits_{\eta\to\infty}\frac{\frac{X_2}{Y_2^2}}{\frac{X_3}{Y_2Y_3}}=\frac{4-2 \mu^2+4m \nu^4}{(4m+8)-2 \mu^2\nu^2-4\nu^2}.
\end{split}
\end{equation}
Solving the equations above yields solutions 
$(1,1)$ and $\left(1,\frac{1}{2m+3}\right)$.

The case where $\mu=0$ does not introduce any difficulty. Limits for $b\dot{b}$ and $c\dot{c}$ are still positive. We have
\begin{equation}
\label{eqn_cigar paraboloid asymtptotics}
\begin{split}
b^2&\sim \frac{2}{-C}\left(4+4m \nu^4\right)t,\quad c^2\sim   \frac{2}{-C}\left((4m+8)-4\nu^2\right)t.
\end{split}
\end{equation}
The limit $\lim\limits_{\eta\to\infty}\frac{X_2}{X_3}=1$ still holds. Solving
\begin{equation}
\label{eqn_computing nu}
\nu^2=\lim\limits_{\eta\to\infty}\frac{\frac{X_2}{Y_2^2}}{\frac{X_3}{Y_2Y_3}}=\frac{4+4m \nu^4}{(4m+8)-4\nu^2}.
\end{equation}
yields solutions $(0,1)$ and $\left(0,\frac{1}{m+1}\right)$.
\end{proof}

By the analysis above, if $\mu\neq 0$, the Ricci soliton is asymptotically paraboloidal whose base is either the standard sphere ($(\mu^2,\nu^2)=\left(1,1\right)$) or the Jensen sphere ($(\mu^2,\nu^2)=\left(1,\frac{1}{2m+3}\right)$). If $\mu=0$, the metric component $a$ for the $\mathbb{S}^1$-fiber does not grow as fast as the other. More specifically, the function converges to a constant. 
%
%

\begin{proposition}
\label{prop_mu=0 is cigar paraboloid}
If $\mu=0$, the limit $\lim\limits_{t\to\infty} a$ exists.
\end{proposition}
\begin{proof}
If $b\equiv c$, Lemma 6.9 from \cite{appleton_family_2017} directly applies. Appleton's approach can be generalized to the case where $b\not\equiv c$ without difficulties. For the sake of completeness, we present the full proof below.

By \eqref{eqn_cohomogeneity_one_soliton_equation}, we have 
\begin{equation}
\begin{split}
\ddot{Y_1}&=\dot{Y_1}\left(\dot{f}-3\frac{\dot{b}}{b}-4m\frac{\dot{c}}{c}\right)-Y_1(1-Y_1^2)\frac{1}{b^2}\left(4+4m\frac{b^4}{c^4}\right)
\end{split}
\end{equation}
Since the factor $\left(\dot{f}-3\frac{\dot{b}}{b}-4m\frac{\dot{c}}{c}\right)$ converges to $-\sqrt{-C}$, for any sufficiently small $\delta>0$, the inequality
\begin{equation}
\begin{split}
\ddot{Y_1}&\leq (-\sqrt{-C}+\delta)\dot{Y_1}-Y_1(1-Y_1^2)\frac{1}{b^2}\left(4+4m\frac{b^4}{c^4}\right)\\
\end{split}
\end{equation}
eventually holds. By \eqref{eqn_cigar paraboloid asymtptotics} and \eqref{eqn_computing nu}, the inequality
$$\frac{1}{b^2}\left(4+4m\frac{b^4}{c^4}\right)\geq \frac{\sqrt{-C}-\delta}{t}$$
eventually holds. Given that $\lim\limits_{\eta\to\infty} Y_1=0$, the inequality
\begin{equation}
\label{eqn_second deri of Y1}
\begin{split}
\ddot{Y_1}&\leq (-\sqrt{-C}+\delta)\dot{Y_1}-Y_1(1-Y_1^2)\frac{\sqrt{-C}-\delta}{t}\\
&\leq (-\sqrt{-C}+\delta)\dot{Y_1}-Y_1\frac{\sqrt{-C}-\delta}{2t}
\end{split}
\end{equation}
holds eventually. Multiply \eqref{eqn_second deri of Y1} by $e^{(\sqrt{-C}-\delta)t}$ and integrate, we obtain 
\begin{equation}
\begin{split}
\dot{Y_1}&\leq -e^{-(\sqrt{-C}-\delta)t}\int_{t_0}^t Y_1(\tilde{t})\frac{\sqrt{-C}-\delta}{2\tilde{t}} e^{(\sqrt{-C}-\delta)\tilde{t}} d\tilde{t}+\dot{Y_1}(t_0)e^{(\sqrt{-C}-\delta)(t_0-t)}.
\end{split}
\end{equation}
Since $Y_1$ is decreasing, the inequality 
\begin{equation}
\begin{split}
\dot{Y_1} &\leq -e^{-(\sqrt{-C}-\delta)t}\frac{Y_1}{t}\int_{t_0}^t \frac{\sqrt{-C}-\delta}{2} e^{(\sqrt{-C}-\delta)\tilde{t}} d\tilde{t}+\dot{Y_1}(t_0)e^{(\sqrt{-C}-\delta)(t_0-t)}\\
&= -\frac{Y_1}{2t}    +\left(\dot{Y_1}(t_0)+\frac{Y_1}{2t}\right)e^{(\sqrt{-C}-\delta)(t_0-t)}\\
&\leq -\frac{Y_1}{2t} \left(1-e^{(\sqrt{-C}-\delta)(t_0-t)}\right)
\end{split}
\end{equation}
eventually holds. As the second term converges to zero, the inequality 
$$Y_1\leq K_0 t^{-\frac{(1-\delta)}{2}}$$
eventually holds for some constant $K_0>0$. By \eqref{eqn_cohomogeneity_one_soliton_equation}, we have 
\begin{equation}
\label{eqn_second deri of a}
\begin{split}
\ddot{a}\leq \dot{a}\dot{f}+Y_1^4\left(\frac{2}{a}+\frac{4m}{a}\frac{b^4}{c^4}\right)\leq -K_1 \dot{a}+\frac{K_2}{t^{2(1-\delta)}}
\end{split}
\end{equation}
for some constants $K_1,K_2>0$. Multiplying \eqref{eqn_second deri of a} by $e^{K_1t}$ and integrating, we have 
$$
\dot{a}\leq K_2 e^{-K_1 t}\int_{t_0}^t \frac{e^{K_1\tilde{t}}}{\tilde{t}^{2(1-\delta)}} d\tilde{t}+\dot{a}(t_0)e^{K_1(t_0-t)}.
$$
By the L'H\^opital’s Rule, it follows that
$$
K_2 e^{-K_1t}\int_{t_0}^t \frac{e^{K_1\tilde{t}}}{\tilde{t}^{2(1-\delta)}} d\tilde{t}\sim \frac{K_2}{K_1t^{2(1-\delta)}}
$$
as $t\to \infty$. Therefore, the inequality $\dot{a}\leq \frac{K_3}{t^{2(1-\delta)}}$ holds eventually for some constant $K_3>0$. The function $a$ is bounded above. Since $X_1\geq 0$ in $\mathcal{A}$, the function $a$ is non-decreasing. Therefore, the limit $\lim\limits_{t\to\infty} a$ exists.
\end{proof}

Therefore, if $\mu=0$, the Ricci soliton is asymptotically a circle bundle with a constant radius over a paraboloid whose base is $\mathbb{CP}^{2m+1}$. If $\nu^2=1$, the base is the Fubini--Study $\mathbb{CP}^{2m+1}$. If $\nu^2=\frac{1}{m+1}$, the base is the homogeneous non-K\"ahler $\mathbb{CP}^{2m+1}$ in \cite{ziller_homogeneous_1982}.

\begin{proposition}
\label{prop_Y3/Y2<2 means ACP}
If $\frac{Y_3}{Y_2}<1$ initially along integral curves in Lemma \ref{lem_long existing_zeta_gamma}, then $\nu^2\neq 1$. If $Y_1<1$ initially along integral curves in Lemma \ref{lem_long existing_zeta_gamma}, then $\mu^2\neq 1$.
\end{proposition}
\begin{proof}
Suppose $\frac{Y_3}{Y_2}<1$ initially and $\lim\limits_{\eta\to\infty}\frac{Y_3}{Y_2}=1$.
By the derivative $(X_2-X_3)'$, if $X_2-X_3$ changes sign infinitely many times, so does $R_2-R_3$. Therefore, if $R_2-R_3$ has a sign eventually, then so does $X_2-X_3$. As $\frac{Y_3}{Y_2}$ converges to $1$, the function $R_2-R_3$ is eventually negative. Since $1\geq \frac{Y_3}{Y_2}$ and $\left(\frac{Y_3}{Y_2}\right)'=2\frac{Y_3}{Y_2}(X_2-X_3)$, the function $X_2-X_3$ is eventually positive.

By \eqref{eqn_key_computation1}, we have 
\begin{equation}
\label{eqn_from key_computation1}
\begin{split}
&\left(2(Y_2-Y_3)+X_3-X_2\right)'\\
&\geq  \left(2(Y_2-Y_3)+X_3-X_2\right)\left(G-1+4 Y_3\right) +2(Y_2-Y_3)\left(1-X_2+2(m-1) Y_3-2 Y_2\right).\\
\end{split}
\end{equation}

By Proposition \ref{prop_points with Y1 nonzero}, we know that $G-1+4Y_3$ is eventually negative. Then \eqref{eqn_from key_computation1} is eventually larger than or equal to
\begin{equation}
\label{eqn_from key_computation3}
\begin{split}
&2(Y_2-Y_3)\left(G-1+4 Y_3\right) +2(Y_2-Y_3)\left(1-X_2+2(m-1) Y_3-2 Y_2\right)\\
&=2(Y_2-Y_3)\left(G-\frac{X_2}{Y_2^2}Y_2^2+ 2(m+1)Y_3-2Y_2\right).
\end{split}
\end{equation}
Based on our assumption that $\nu=1$, the inequality $2(m+1)Y_3-2Y_2\geq Y_2$ eventually holds. Moreover, by \eqref{eqn_limit of X2/Y2^2}, the inequality $\frac{X_2}{Y_2^2}\leq 4m+4$ eventually holds. Hence, the above computation is eventually larger than or equal to
\begin{equation}
\label{eqn_from key_computation4}
\begin{split}
2(Y_2-Y_3)\left( -\frac{X_2}{Y_2^2}Y_2^2+Y_2\right)&\geq 2(Y_2-Y_3)\left( -(4m+4)Y_2^2+Y_2\right)\\
&= 2(Y_2-Y_3)Y_2(1-(4m+4)Y_2).
\end{split}
\end{equation}
By Proposition \ref{prop_points with Y1 nonzero}, the derivative $\left(2(Y_2-Y_3)+X_3-X_2\right)'$ is eventually positive. As the function $2(Y_2-Y_3)+X_3-X_2$ converges to zero and the integral curve remains in the invariant set $\mathcal{A}$ by Lemma \ref{lem_invariant_setA}, we reach a contradiction. Thus, the first statement holds.

Since $Y_1$ is non-increasing along $\zeta(s_1,s_2,0,s_4)$ and $\gamma(s_1,s_2,0,s_4)$, the second statement holds. The proof is complete.
\end{proof}

\begin{lemma}
\label{lem_asymp for epsilon=0}
The asymptotics for integral curves with $s_3=0$ in Lemma \ref{lem_long existing_zeta_gamma} are listed below.
\end{lemma}
\begin{table}[H]
\centering
  \begin{tabular}{l l l}
   Integral curves &  Asymptotics & Base  \\
\hline
\hline
   $\zeta(s_1,0,0,s_4)$ & AP & Jensen $\mathbb{S}^{4m+3}$\\ 
   $\zeta(s_1,s_2,0,s_4)$ & ACP & non-K\"ahler $\mathbb{CP}^{2m+1}$\\ 
       \hline
   $\gamma(0,0,0,1)$ & AP & standard $\mathbb{S}^{4m+3}$\\ 
   $\gamma(s_1,0,0,s_4)$ & AP & Jensen $\mathbb{S}^{4m+3}$\\ 
   $\gamma(0,s_2,0,s_4)$ & ACP & Fubini--Study $\mathbb{CP}^{2m+1}$\\ 
   $\gamma(s_1,s_2,0,s_4)$ & ACP & non-K\"ahler $\mathbb{CP}^{2m+1}$\\ 
       \hline
  \end{tabular}
  \caption{}
  \label{tab_asymp1}
\end{table}
\begin{proof}
The four possible asymptotics for $\gamma$ correspond to the four possible values for $(\mu^2,\nu^2)$ in Lemma \ref{lem_possible asymptotics}. With Proposition \ref{prop_Y3/Y2<2 means ACP} and the monotonicity of $Y_1^2$, the asymptotics are clear depending on whether the integral curve $\gamma$ stays in $\mathcal{RS}_{\text{round}}$ or $\mathcal{R}_{\text{FS}}$. The proof for $\zeta$ is similar. As we set $s_1>0$ for $\zeta$ so that $Y_3$ is positive, there are only two possible asymptotics for $\zeta$ depending on whether or not the integral curve stays in $\mathcal{RS}_{\text{round}}$. 
\end{proof}

\subsection{Expanding Ricci solitons}
In this section, we follow \cite[Proposition 2.20]{wink_complete_2021} to prove the asymptotic for expanding Ricci solitons. 

\begin{proposition}
\label{prop_limit of X/W}
Along each integral curve in Lemma \ref{lem_long existing_zeta_gamma} with $s_3, s_4>0$, each function $\frac{X_i}{W}$ converges to $\frac{\epsilon}{2}$. 
\end{proposition}
\begin{proof}
By Proposition \ref{prop_points with Y1 nonzero}, all coordinate functions except $Y_1$ converge to zero as $\eta\to \infty$. From the derivatives
\begin{equation}
\begin{split}
\left(\frac{Y_2^2}{Y_1^2W}\right)'=-2\frac{Y_2^2}{Y_1^2W}X_1,\quad \left(\frac{Y_2^2}{W}\right)'=-2\frac{Y_2^2}{W}X_2,\quad \left(\frac{Y_2Y_3}{W}\right)'=-2\frac{Y_2Y_3}{W}X_3,
\end{split}
\end{equation}
positive functions $\frac{Y_2^2}{Y_1^2W}$, $\frac{Y_2^2}{W}$ and $\frac{Y_2Y_3}{W}$ monotonically decrease and hence converge. 

Consider
\begin{equation}
\left(\frac{X_i}{W}\right)'=\frac{X_i}{W}\left(-G+\frac{\epsilon}{2}W-1\right)+\frac{R_i}{W}+\frac{\epsilon}{2}.
\end{equation}
If $\lim\limits_{\eta\to\infty}\frac{Y_2^2}{W}=0$, each $\frac{R_i}{W}$ converges to zero as $\frac{Y_3}{Y_2}$ is bounded. Therefore, each $\frac{X_i}{W}$ converges to $\frac{\epsilon}{2}$ by Proposition \ref{prop_model_ode}. If $\frac{Y_2^2}{W}$ converges to some positive number, then $\frac{Y_3}{Y_2}=\frac{Y_2Y_3}{W}\frac{W}{Y_2^2}$ also converges. Thus, each $\frac{R_i}{W}$ and each $\frac{X_i}{W}$ converges by Proposition \ref{prop_model_ode}. We need to show that $\lim\limits_{\eta\to\infty}\frac{X_i}{W}=\frac{\epsilon}{2}$ for this case.

Since each $X_i$ converges to zero, we must have 
$$\lim\limits_{\eta\to\infty} \frac{G}{W}=\lim\limits_{\eta\to\infty}\left(X_1\frac{X_1}{W}+2X_2\frac{X_2}{W}+4mX_3\frac{X_3}{W}\right)=0.$$
Assume $\frac{X_i}{W}$ converges to some $\frac{\epsilon}{2}+\delta$ with $\delta\neq 0$. Take $i=2$, for example. We have 
$$
\lim\limits_{\eta\to\infty}\frac{(Y_2^2)'}{W'}=\lim\limits_{\eta\to\infty}\frac{Y_2^2(G-\frac{\epsilon}{2}W-X_2)}{W(G-\frac{\epsilon}{2}W)}=\lim\limits_{\eta\to\infty}\frac{Y_2^2\left(\frac{G}{W}-\frac{\epsilon}{2}-\frac{X_2}{W}\right)}{W\left(\frac{G}{W}-\frac{\epsilon}{2}\right)}=\lim\limits_{\eta\to\infty}\frac{Y_2^2}{W}\frac{\epsilon+2\delta}{\epsilon}.
$$
The above equation violates the L'H\^opital's rule. Thus, the limit must be $\frac{\epsilon}{2}$. 
\end{proof}

\begin{lemma}
\label{lem_asymp for epsilon=1}
Non-Einstein expanding Ricci solitons represented by $\zeta(s_1,s_2,s_3,s_4)$ and $\gamma(s_1,s_2,s_3,s_4)$ in Lemma \ref{lem_long existing_zeta_gamma} are AC.
\end{lemma}
\begin{proof}
Since $W'=2W\left(G-\frac{\epsilon}{2}W\right)=2W^2\left(\frac{G}{W}-\frac{\epsilon}{2}\right)$ and $\lim\limits_{\eta\to\infty} \frac{G}{W}=0$, for any small enough $\delta>0$, there exists $\eta_*$ such that $\eta>\eta_*$ implies $2W^2\left(-\delta-\frac{\epsilon}{2}\right)\leq W'\leq 2W^2\left(\delta-\frac{\epsilon}{2}\right)$. By integrating, we obtain
$$\frac{W(\eta_*)}{1+W(\eta_*)(\epsilon+2\delta)(\eta-\eta_*)}\leq W\leq \frac{W(\eta_*)}{1+W(\eta_*)(\epsilon-2\delta)(\eta-\eta_*)}.$$
Therefore, we conclude that $\lim\limits_{\eta\to\infty} W\eta=\frac{1}{\epsilon}$. By Proposition \ref{prop_limit of X/W}, we have  
\begin{equation}
\label{eqn_limit of X/W 2}
\frac{X_i}{\sqrt{W}}\sim \frac{\epsilon}{2}\sqrt{W}\sim \frac{1}{2}\sqrt{\frac{\epsilon}{\eta}}
\end{equation}
as $\eta\to \infty$.
As $dt=\sqrt{W}d\eta$, it is clear that $\eta\sim\frac{\epsilon}{4}t^2$. It follows that
$
\frac{\dot{a}}{a},\frac{\dot{b}}{b},\frac{\dot{c}}{c}\sim \frac{1}{t}
$
as $t\to\infty$. The conical asymptotic is established.
\end{proof}

Lemma \ref{lem_long existing_zeta_gamma}, Lemma \ref{lem_asymp for epsilon=0} and Lemma \ref{lem_asymp for epsilon=1} prove Theorem \ref{thm_1} and Theorem \ref{thm_2}.

\section{Ricci solitons on $\mathbb{O}^2$}
If the group triple is 
$$
(\mathsf{K},\mathsf{H},\mathsf{G})=(\spin(7),\spin(8),\spin(9)),
$$
the principal orbit is the total space of the octonionic Hopf fibration.
The associated cohomogeneity one equation is analogous to the subsystem of \eqref{eqn_cohomogeneity_one_soliton_equation} with $a=b$. Specifically, the cohomogeneity one equation is
\begin{equation}
\label{eqn_cohomogeneity_one_soliton_equation (2)}
\begin{split}
\frac{\ddot{b}}{b}-\left(\frac{\dot{b}}{b}\right)^2&=-\left(7\frac{\dot{b}}{b}+8\frac{\dot{c}}{c}\right)\frac{\dot{b}}{b}+\frac{6}{b^2}+8\frac{b^2}{c^4}+\frac{\dot{b}}{b}\dot{f}+\frac{\epsilon}{2},\\
\frac{\ddot{c}}{c}-\left(\frac{\dot{c}}{c}\right)^2&=-\left(7\frac{\dot{b}}{b}+8\frac{\dot{c}}{c}\right)\frac{\dot{c}}{c}+\frac{28}{c^2}-14\frac{b^2}{c^4}+\frac{\dot{c}}{c}\dot{f}+\frac{\epsilon}{2},\\
\ddot{f}+\frac{\epsilon}{2}&=7\frac{\ddot{b}}{b}+8\frac{\ddot{c}}{c},
\end{split}
\end{equation}
where $b$ and $c$ are metric components for the $\mathbb{S}^7$-fiber and the base space $\mathbb{OP}^1$, respectively. The initial condition where the principal orbit collapses to $\mathbb{OP}^1$ is 
\begin{equation}
\label{eqn_OP1}
\lim\limits_{t\to 0}(b,c,\dot{b},\dot{c})=(0,c_0,1,0),\quad c_0>0.
\end{equation}
If $\mathbb{S}^{15}$ fully collapses, the initial condition is
\begin{equation}
\label{eqn_S15 collapse}
\lim\limits_{t\to 0}(b,c,\dot{b},\dot{c})=(0,0,1,1)
\end{equation}
with a free parameter from $(\dddot{b}-\dddot{c})(0)$.

Apply the coordinate change as in \eqref{eqn_new_variable_for_new_coordinate}, we have 
\begin{equation}
\label{eqn_S15 cohomo1 equation}
\begin{bmatrix}
X_2\\
X_3\\
Y_2\\
Y_3\\
W
\end{bmatrix}'=V(X_2,X_3,Y_2,Y_3,W)=\begin{bmatrix}
X_2\left(G-\frac{\epsilon}{2}W-1\right)+R_2+\frac{\epsilon}{2}W\\
X_3\left(G-\frac{\epsilon}{2}W-1\right)+R_3+\frac{\epsilon}{2}W\\
Y_2\left(G-\frac{\epsilon}{2}W-X_2\right)\\
Y_3\left(G-\frac{\epsilon}{2}W+X_2-2X_3\right)\\
2W\left(G-\frac{\epsilon}{2}W\right)
\end{bmatrix}
\end{equation}
where
\begin{equation}
\begin{split}
R_2&=6Y_2^2+8Y_3^2,\quad R_3=28Y_2Y_3-14Y_3^2,\quad R_s=7R_2+8R_3,\\
H&=7X_2+8X_3,\quad G=7X_2^2+8X_3^2.
\end{split}
\end{equation}
Define $Q:=G+R_s+7\epsilon W-1$. The dynamical system \eqref{eqn_S15 cohomo1 equation} on the invariant set
$$
\{Q\leq 0\}\cap\{H\leq 1\}\cap \{W\geq 0\}\cap \{Y_2,Y_3\geq 0\}
$$
is analogous to \eqref{eqn_Polynomial_soliton_equation} restricted on $\mathcal{RS}_{\text{round}}$. Initial conditions \eqref{eqn_OP1} and \eqref{eqn_S15 collapse} are respectively transformed to critical points
\begin{equation}
\label{eqn_linearized_tildegamma}
\left(\frac{1}{7},0,\frac{1}{7},0,0\right),\quad \left(\frac{1}{15},\frac{1}{15},\frac{1}{15},\frac{1}{15},0\right).
\end{equation}

Define
\begin{equation}
\label{eqn_settildeA}
\begin{split}
\tilde{\mathcal{A}}&:=\mathcal{RS}\cap\{Y_2-Y_3\geq 0\}\cap \left\{2(Y_2-Y_3)+X_3-X_2\geq 0\right\} \cap \{X_2,X_3\geq 0\}.
\end{split}
\end{equation}

\begin{lemma}
The set $\tilde{\mathcal{A}}$ is compact and invariant.
\end{lemma}
\begin{proof}
As $Y_2$ is bounded above, all other variables are bounded above by $Y_2\geq Y_3$ and $Q\leq 0$. Thus, the set $\tilde{\mathcal{A}}$ is compact.

Since $Y_2-Y_3\geq 0$, we have $R_3\geq 0$. Therefore, the vector field $V$ on each face $\tilde{\mathcal{A}}\cap \{X_i=0\}$ points inward. The computation for the face $\tilde{\mathcal{A}}\cap \{Y_2-Y_3=0\}$ is similar to \eqref{eqn_key_computation1}. We focus on showing that $V$ restricted on $\tilde{\mathcal{A}}\cap \{2(Y_2-Y_3)+X_3-X_2= 0\}$ points inward.

From $H\leq 1$, it is clear that $X_2\leq \frac{1}{7}$ and $X_3\leq \frac{1}{8}$ hold in $\tilde{\mathcal{A}}$. We have 
\begin{equation}
\begin{split}
&(2(Y_2-Y_3)+X_3-X_2)'\\
&=(2(Y_2-Y_3)+X_3-X_2)\left(G-\frac{\epsilon}{2}W-1+4Y_3\right)+2(Y_2-Y_3)(1-X_2-3Y_2+7Y_3)\\
&\geq (2(Y_2-Y_3)+X_3-X_2)\left(G-\frac{\epsilon}{2}W-1+4Y_3\right)+2(Y_2-Y_3)\left(\frac{6}{7}-3Y_2+7Y_3\right).
\end{split}
\end{equation}
By $Q\leq 0$, we have $1\geq 42 Y_2^2$ on $\tilde{\mathcal{A}}\cap \{2(Y_2-Y_3)+X_3-X_2=0\}$. Since $\frac{4}{49}>\frac{1}{42}\geq Y_2^2$, we conclude that $\frac{6}{7}-3Y_2+7Y_3>0$ on $\tilde{\mathcal{A}}\cap \{2(Y_2-Y_3)+X_3-X_2=0\}$. Therefore, the vector field $\tilde{V}$ restricted on all faces of $\tilde{\mathcal{A}}$ points inward. The set $\tilde{\mathcal{A}}$ is invariant.
\end{proof}

By the linearization of the critical point $\left(\frac{1}{7},0,\frac{1}{7},0,0\right)$, there exists a 2-parameter family of integral curves in $\tilde{\mathcal{A}}$. The Wink solitons on $\mathbb{OP}^2\backslash\{*\}$ are recovered.

On the other hand, there exists a 2-parameter family of integral curves, denoted as $\tilde{\gamma}(s_1,s_3,s_4)$, that emanate the critical point $\left(\frac{1}{15},\frac{1}{15},\frac{1}{15},\frac{1}{15},0\right)$. The $\tilde{\gamma}(s_1,s_3,s_4)$ family is analogous to the $\gamma(s_1,0,s_3,s_4)$ family, with each parameter sharing the same geometric meaning. In particular, the third order free parameter in \eqref{eqn_S15 collapse} is non-zero if $s_1>0$. The asymptotic analysis is simpler for this case. As the $\mathbb{S}^7$-fiber must be round, Proposition \ref{prop_b/c limit exists} carries over, and the limit $\nu$ must exist. We either have $\nu=1$ that gives the standard $\mathbb{S}^{15}$, or $\nu=\frac{3}{11}$ that gives the Bourguignon--Karcher $\mathbb{S}^{15}$. Proposition \ref{prop_Y3/Y2<2 means ACP} is easily generalized. Theorem \ref{thm_3} is established.

\section{Solitons with positive sectional curvatures}
In this section, we consider the subsystem restricted to $\mathcal{RS}_{\text{steady}}\cap \mathcal{RS}_{\text{round}}$, where $X_1=X_2$ and $Y_1=1$. The sectional curvature consists of six components. In the new coordinates, sectional curvatures are expressed as polynomials of $(X_i,Y_i)$. Specifically, we have sectional curvatures for:
\begin{enumerate}
\item
radial--normal 2-planes
$$K_{t2}=X_2-X_2^2-2Y_2^2-4mY_3^2,\quad K_{t3}=X_3-X_3^2-(4m+8)Y_2Y_3+6 Y_3^2,$$
\item
fiber--fiber 2-planes
$$K_{22}=Y_2^2-X_2^2$$
\item
fiber--horizontal 2-planes
$$K_{23}=Y_3^2-X_2X_3$$
\item
horizontal--horizontal 2-planes
$$K_{33}^\mathbb{H}=4Y_2Y_3-3Y_3^2-X_3^2,\quad K_{33}^\mathbb{R}=Y_2Y_3-X_3^2$$
\end{enumerate}
Since the Bryant soliton has positive sectional curvature, the integral curve $\gamma(0,0,0,1)$ is contained in the interior of the following set
$$
\mathcal{K}:=\mathcal{RS}_{\text{steady}}\cap \mathcal{RS}_{\text{round}}\cap \{K_{t2},K_{t3},K_{22},K_{33}, K_{33}^\mathbb{H}, K_{33}^\mathbb{R}>0\}.
$$
\begin{proposition}
\label{prop_forward invariant set}
There exists a sufficiently large $\eta_+$ and a sufficiently small $\delta_1>0$, such that for each $s_1\in [0,\delta_1)$ the integral curve $\gamma(s_1,0,0,\sqrt{1-s_1^2})$ enters $\mathcal{K}$ and remains in the set for $\eta\geq \eta_+$.
\end{proposition}
\begin{proof}
By \eqref{eqn_limit of X2/Y2^2} and Proposition \ref{prop_points with Y1 nonzero}, for each $s_1\in [0,1)$ there exists a corresponding $\eta_{s_1}^+$ such that 
\begin{enumerate}
\item
$K_{22},K_{33}, K_{33}^\mathbb{H}, K_{33}^\mathbb{R}>0$
\item
$1+16m\frac{Y_3^2}{X_2}\frac{X_3}{X_2}-16m\frac{Y_3^2}{X_2}-\frac{G}{X_2}>0$
\item
$1+12\frac{Y_3^2}{X_3}\frac{X_2}{X_3}-12\frac{Y_3^2}{X_3}-\frac{G}{X_3}>0$
\end{enumerate}
along $\gamma(s_1,0,0,\sqrt{1-s_1^2})$ for $\eta\geq \eta_{s_1}^+$. Without loss of generality, make $\eta_{s_1}^+$ a continuous function of $s_1$.

Define $\eta_+=\max\limits_{s_1\in\left[0,\frac{1}{2}\right]} \{\eta_{s_1}^+\}$. Since $\gamma(0,0,0,1)$ is contained in $\mathcal{K}$, there exists a sufficiently small $\delta_1\in\left(0,\frac{1}{2}\right)$ such that $K_{t2},K_{t3}>0$ at $(\gamma(s_1,0,0,\sqrt{1-s_1^2}))(\eta_+)$ for $s_1\in [0,\delta_1)$. Since
\begin{equation}
\label{eqn_Kti}
\begin{split}
K_{t2}'&=K_{t2}(2G-1)+X_2^2\left(1+16m\frac{Y_3^2}{X_2}\frac{X_3}{X_2}-16m\frac{Y_3^2}{X_2}-\frac{G}{X_2}\right),\\
K_{t3}'&=K_{t3}(2G-1)+X_3^2\left(1+12\frac{Y_3^2}{X_3}\frac{X_2}{X_3}-12\frac{Y_3^2}{X_3}-\frac{G}{X_3}\right),
\end{split}
\end{equation}
the signs of $K_{t2}$ and $K_{t3}$ do not change along $\gamma(s_1,0,0,\sqrt{1-s_1^2})$ for $\eta\geq \eta_+$. Therefore, each $\gamma(s_1,0,0,\sqrt{1-s_1^2})$ with $s_1\in[0,\delta_1)$ enters $\mathcal{K}$ and remains in the set for $\eta\geq \eta_+$.
\end{proof}

\begin{proposition}
\label{prop_backward invariant set}
There exists a sufficiently small $\eta_-$ such that for each $s_1\in \left[0,\frac{1}{\sqrt{85}}\right]$ the integral curve $\gamma(s_1,0,0,\sqrt{1-s_1^2})$ is in $\mathcal{K}$ for $\eta\leq \eta_-$.
\end{proposition}
\begin{proof}
Each $\gamma(s_1,0,0,\sqrt{1-s_1^2})$ is initially in $\mathcal{K}$ if $s_1\in\left[0,\frac{2}{\sqrt{85}}\right)$ by \eqref{eqn_linearized gamma negative einstein}. Hence, for each $s_1\in \left[0,\frac{2}{\sqrt{85}}\right)$, there exists a corresponding $\eta_{s_1}^-$ such that $\gamma(s_1,0,0,\sqrt{1-s_1^2})$ is in $\mathcal{K}$ for $\eta \leq \eta_{s_1}^-$.  Without loss of generality, make $\eta_{s_1}^-$ a continuous function of $s_1$. The proof is complete by defining $\eta_-=\min\limits_{s_1\in\left[0,\frac{1}{\sqrt{85}}\right]} \{\eta_{s_1}^-\}$.
\end{proof}

\begin{lemma}
\label{lem_positive sectional curvature assemble}
There exists a sufficiently small $\delta>0$ such that each $\gamma(s_1,0,0,\sqrt{1-s_1^2})$ with $s_1\in[0,\delta)$ is contained in $\mathcal{K}$.
\end{lemma}
\begin{proof}
Fix $\eta_+$ as in the proof of Proposition \ref{prop_forward invariant set} and $\eta^-$ as in Proposition \ref{prop_backward invariant set}. The graph $(\gamma(0,0,0,1))([\eta_-,\eta_+])$ is a compact subset of $\mathcal{K}$. There exists an open (w.r.t. $\mathcal{RS}_{\text{steady}}\cap \mathcal{RS}_{\text{round}}$) $\epsilon$-tube $\mathcal{T}_\epsilon\subset\mathcal{K}$ around $(\gamma(0,0,0,1))([\eta_-,\eta_+])$. By the continuous dependence, there exists a sufficiently small $\delta_2>0$ such that each $\gamma(s_1,0,0,\sqrt{1-s_1^2})$ with $s_1\in[0,\delta_2)$ remains in $\mathcal{T}_\epsilon$ for $\eta\in[\eta_-,\eta_+]$. Let $\delta=\min\left\{\delta_1,\delta_2,\frac{1}{\sqrt{85}}\right\}$, then each $\gamma(s_1,0,0,\sqrt{1-s_1^2})$ with $s_1\in[0,\delta)$ is contained in $\mathcal{K}$.
\end{proof}

The above argument can be carried over to the $\mathbb{O}^2$ case. Theorem \ref{thm_4} is therefore established.

\section{Visual Summary}
We use Grapher to make the following figures. By $s_4=\sqrt{1-(s_1^2+s_2^2+s_3)}$, we use points in the $s_1s_2s_3$-space to represent Ricci solitons on $\mathbb{HP}^{m+1}\backslash\{*\}$ and $\mathbb{H}^{m+1}$. 
\begin{figure}[H]
  \centering
  \includegraphics[width=1\linewidth]{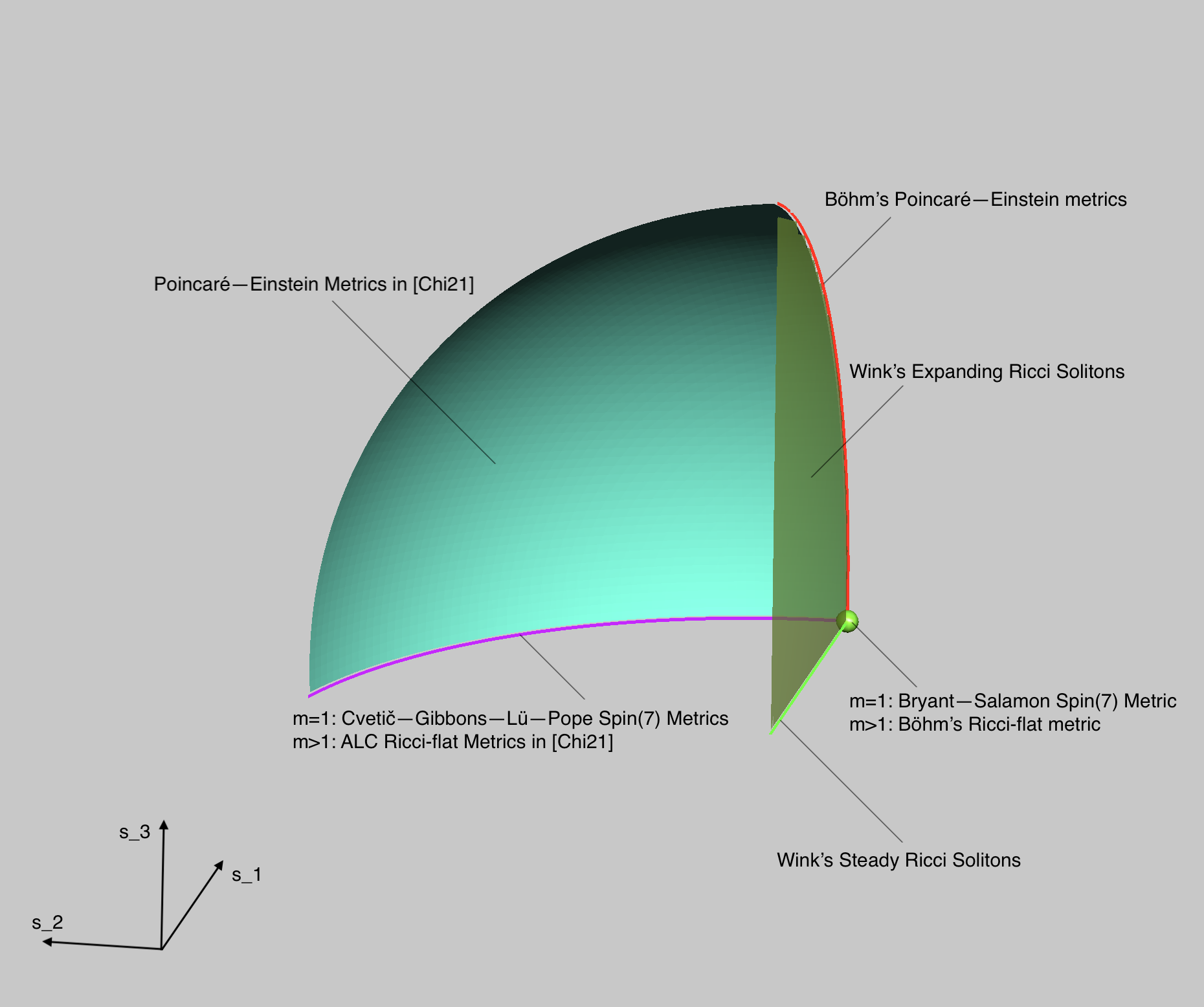}
\caption{Ricci solitons on $\mathbb{HP}^{m+1}\backslash\{*\}$}
\label{fig_1}
\end{figure}
\begin{figure}[H]
  \centering
  \includegraphics[width=1\linewidth]{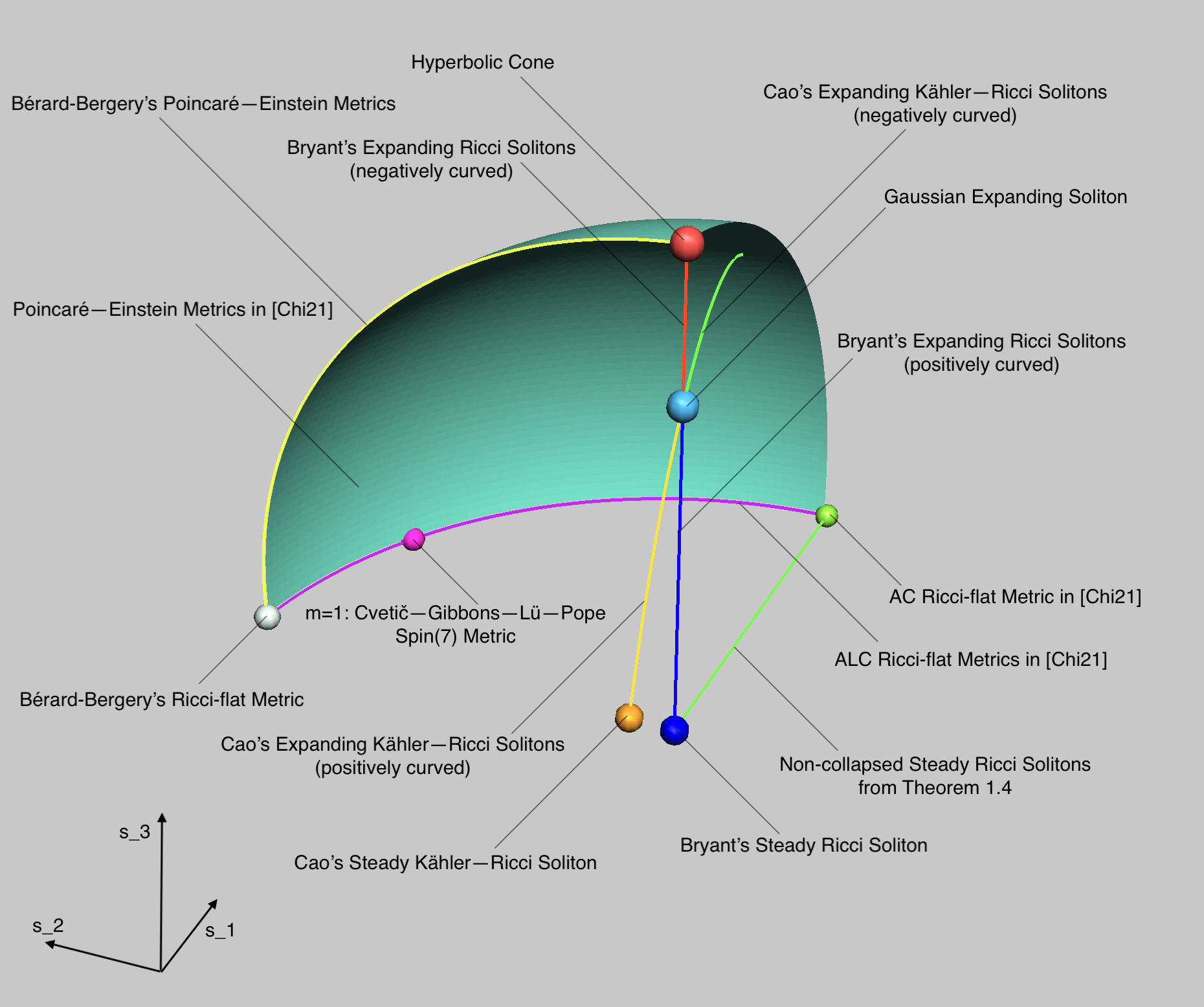}
\caption{Ricci solitons on $\mathbb{H}^{m+1}$}
\label{fig_2}
\end{figure}

\section*{Acknowledgements}
The author is grateful to McKenzie Wang for valuable discussions on this project. The author thanks Xiping Zhu for suggesting this topic, which extends the existing research on cohomogeneity one Einstein metrics. The author also thanks Huaidong Cao for his comments on Kähler--Ricci solitons.

\bibliography{Bibliography}
\bibliographystyle{alpha}

\end{document}